\documentclass[preprint]{article}
\usepackage{amsthm,amsmath,amssymb, algorithm, algpseudocode, bm, graphicx, dsfont,epstopdf, authblk}
\RequirePackage[colorlinks,citecolor=blue,urlcolor=blue]{hyperref}
 \usepackage[section]{placeins}

\newcommand{\dif}{\mathop{}\!\mathrm{d}}

\newtheorem{prop}{Proposition}
\newtheorem{lem}{Lemma}
\newtheorem{cor}{Corollary}
\theoremstyle{definition}
\newtheorem{defn}{Definition}

\newtheorem{rmk}{Remark}

\newcommand{\D}{ \mathcal{ D } }

\newcommand{\F}{ \mathcal{ F } }

\newcommand{\G}{ \mathcal{ G } }

\newcommand{\N}{ \mathbb{ N } }
\newcommand{\E}{ \mathbb{ E } }
\newcommand{\R}{ \mathbb{ R } }
\newcommand{\Prb}{ \mathbb{ P } }

\newcommand{\Z}{ \mathbb{ Z } }

\newcommand{\totalm}{\vert\boldsymbol{m}\vert}

\newcommand{\bs}[1]{\boldsymbol{#1}}
\newcommand{\wt}[1]{\widetilde{#1}}
\newcommand{\Option}[1]{\Statex \textbf{#1}}
\newcommand{\Indent}[1]{\State \hspace{\algorithmicindent} #1}
\makeatother
\newcommand{\EndIndent}{}

\title{Subordinated Wright-Fisher priors}

\author[1]{Nathan A. Judd}

\author[2]{Dario Span\`o}

\affil[1]{Okinawa Institute of Science and Technology, Japan}
\affil[2]{Department of Statistics, University of Warwick, United Kingdom}

\begin{document}

\maketitle

\begin{abstract}
A new class of time-dependent Dirichlet priors is introduced as a generalisation of the Wright-Fisher diffusion, allowing discontinuities in the trajectories, as well as non-Markovian memory. This class is obtained as a simple stochastic time-change (subordination), interpreted as a hyper-prior assigned to the operational time-clock of a Wright-Fisher diffusion. Explicit representations and exact sampling algorithms are obtained for prior and posterior distributions of the process and of its clock, given partially exchangeable data sampled at discrete time-points. Computability and conjugacy rely on a novel class of discrete dual processes, generalising existing results on duality and computable filters.

\vspace{1em}

\vspace{1em}
\end{abstract}

\section{Introduction}

Wright-Fisher (WF) diffusions are a class of continuous-time Markov processes describing the dynamics of $K$ proportions of types. Their state space is the $(K-1)$-dimensional simplex $\Delta_K:=\{{x}\in[0,1]^K : \sum_{i=1}^K{x_i}=1\}$, $K\in\N$. 
A WF diffusion is time-reversible and has a Dirichlet stationary distribution. 
It has been introduced and studied in depth in the context of population genetics, to describe the evolution in time of a population's relative proportions of $K$ allele frequencies, under the assumption of parent-independent mutation (see e.g. \cite{Ewens79} or \cite{SomeMathModels}). Recently, however, the WF diffusion has been increasingly considered for a wider range of applications, including stochastic volatility models in financial econometrics \cite{GOUR06}, neuronal signal models \cite{donofrio_jacobi_fpt}, denoising in machine learning \cite{chandra2025} and more generally in methodological statistics, as a continuous-time-dependent extension of the Dirichlet prior for Bayesian inference and filtering problems \cite{Papaspiliopoulos_2014,kkk, GENONCATALOT2003}. This is also the context considered in this paper:
For any $N\in\N,$ let data $y_{0:N}:=(y_0,...,y_N)$ observed at times $0=t_0<...<t_N$ be modeled as realisations of random vectors $Y_i := Y(t_i)$, sampled from an unobservable parameter - or signal - $({{X}}_t: t\geq 0)$ evolving through time in $\Delta_K$:
\begin{equation}
        \begin{aligned}
        \label{The WF Model}
            Y(t_i) \vert {X}(t_i) &\sim f_{{X}(t_i)} \ \ \ i=1,\ldots, N\\
        ({X}(t):t\geq 0) &\sim \text{WF}({\theta},\nu^*).
        \end{aligned}
\end{equation}
Here $f_{ x}$ is the conditional likelihood (emission density) of the data, given the current state $ x$ of the parameter (signal), and is typically assumed to be either {\em i.i.d.} categorical or multinomial, with parameter $ x$, for a given sample size. The prior law WF(${\theta},\nu^*$) of the stochastic process $ X$, given a \lq\lq mutation\rq\rq\ parameter $\theta=(\theta_1,\ldots,\theta_K)\in(0,\infty)^K,$  and an initial distribution $ X(0)\sim \nu^*$, is fully determined by the associated generator, recalled by formula \eqref{WF gen} in Section \ref{sec wf properties}, and has the Dirichlet($\theta$) distribution as unique stationary distribution.
The model \eqref{The WF Model} has been shown by \cite{Papaspiliopoulos_2014} to possess computable posterior distributions, mapping finite mixtures of Dirichlet to finite mixtures of Dirichlet.\\
As with every diffusion, the trajectories of a WF are almost surely continuous, so it is a sensible choice of prior if it is reasonable to expect no jumps in any infinitesimal time-gap. This assumption however may not hold in every applied context. 
In finance, jump behaviour owed e.g. to shocks \cite{HAMILTON1990}, may be difficult to discard a priori and discontinuities in neuronal signals are also plausible \cite{donofrio_jacobi_fpt}. Equally, the Markov property of a WF might not be adequate to describe phenomena for which it is more reasonable to assume long-range autocorrelation (see e.g. \cite{hawkes1971}, \cite{Ogata01031988}, \cite{Hawkes01022018}, \cite{Tina2006}, \cite{Samor2016}). 

 In this paper a new, large class of time-dependent Dirichlet priors is introduced, retaining much of the tractability of WF diffusions but, at the same time, allowing trajectories performing jumps alongside diffusive behaviour, as well as permitting long-range memory as opposed to Markovian lack of memory. This generalised class, which we call {\em subordinated Wright-Fisher} (sWF) {\em priors}, is obtained by a simple transformation of the WF diffusion, known as {\em subordination} \cite{Bochner49}: an arbitrary stochastic time-change applied through an independent, non-decreasing  stochastic process $C=(C(t):t\geq 0)$, starting from $C(0)=0$. 
The new prior is the simplex-valued process  $\widetilde{{X}}(t):= X\circ C (t)= {{X}}(C(t))$. The family of time-changes $C$ considered here includes: one-dimensional L\'evy subordinators  (i.e. processes with almost-surely non-negative, independent, stationary increments); inverse subordinators (i.e. crossing-time processes of subordinators); any composition of the latter with the former type. 
We will fully characterise the sWF prior in terms of its transition function, generator and resolvent, and derive explicitly the posterior distributions of both $\wt X$ and $C$. We will provide exact simulation algorithms and show how computability of posterior inference (e.g. filter) depends on the choice of $C$.
In particular, we will show that, whenever the initial law of the sWF is a finite mixture of Dirichlet distributions, the posterior distributions of a sWF can also be written as a collection of finite mixtures of Dirichlet distributions with explicit weights. This result hinges on the existence of (1) an infinite-series representation of the sWF's transition function, which inherits its structure from its diffusive counterpart and, (2), a pure-jump stochastic dual process, with non-increasing trajectories, generalizing the dual used in \cite{Papaspiliopoulos_2014}, determining the posterior mixture weights. In this paper we focus on $K$-dimensional Dirichlet but all our results can be readily generalised to a non-parametric setting, to obtain time-dependent Dirichlet Process priors via subordinated Fleming-Viot processes, structurally identical \cite{TFofFV}.
\subsubsection*{Time-change as hyper-prior}
Looking at the setting through a Bayesian lens, choosing a stochastic process to time-change a WF prior $(X(t))$, essentially corresponds to assigning a hyperprior $C$ on the \lq\lq clock\rq\rq\ of the time parameter $t$ of $(X(t))$, which modulates the shape of the discontinuities of $( {\widetilde X}(t))$, and its memory. 
For example, choosing a L\'evy subordinator for $(C(t))$ with characteristic L\'evy triplet $(0,\beta,\pi)$ (see below, Section \ref{sec sub time change}), corresponds to a hyperprior selecting almost surely Markovian memory \cite{Bouleau84, kenlevy}, whereby the drift $\beta\geq 0$ represents the tendency of the prior to evolve as a diffusion (no jumps), and the L\'evy measure $\pi$ governs the frequency and sizes of the possible jumps. The trivial hyperprior $C(t)\equiv t$ will, of course, recover the classical WF diffusion prior.
An inverse subordinator clock, on the other hand, will keep the trajectories continuous but the memory will be non-Markovian: its underlying L\'evy measure $\pi$ will regulate the range and strength of the dependence of the diffusion from its far past. 
Even in such cases, the evolution equation of $\wt X$ can still be described quite explicitly as generalisation of certain classes of fractional kinetic equations (see Section \ref{sec sub inv time change} below; also see e.g. \cite{Kol09}, \cite{Toaldo2015} and references therein). 
The best-known case of non-Markovian time-change is given by the crossing-time process of a stable subordinator, giving rise to Caputo-type fractional evolution equations (\cite{meer2012}, \cite{mainardi2010}). Non-Markovian, discontinuous sWF can be obtained by a two-layer composition of two independent clocks, say $C=R\circ S$, where $R$ an inverse subordinator and $S$ a subordinator. 
Thus, if one is able to operationally control the hyperprior-level clock $(C(t))$, one can then leverage the tractability of the WF diffusion, particularly of its transition function, to describe in deep detail the structure of the time-marginal distributions of the sWF process $\widetilde{ X}$, both {\em a priori} and {\em a posteriori} (conditionally on data), and derive exact likelihood-based algorithms for simulation and inference.
\subsubsection*{Time-dependence structure and stochastic duality}

The model \eqref{The WF Model} assumes exchangeability of categorical observables within the same time, but only partially exchangeability between distinct times, since the underlying parameter $X(t)$ differs almost surely at each time $t$. 
The transition function of the WF diffusion thus plays a crucial role, in that it describes the dependence structure of the parameter $X$ between distinct times, affecting the correlation between non-synchronous samples. 
In this paper, it will be shown (Section \ref{sec time-marginal}) that the time-marginal distributions of the sWF prior admit two series representations, either as an infinite eigenvalue-eigenfunctions series or as an infinite mixture of posterior Dirichlet measures, replicating the structure of the classical WF transition function. In fact, the Laplace transform of $C(t)$ is all one needs to determine the eigenvalues of the evolution operator of $(\wt X(t))$ as a function of the WF eigenvalues. To ensure amenability for exact simulation and inference (see Sections \ref{sec algorithms}-\ref{sec posterior}), the notion of stochastic duality turns out to be crucially helpful in combining effectively the dynamics of the prior $X$ and the hyperprior clock $C$, and obtain a complete description of the evolution of $\wt X=X\circ C$.  \\ 
The WF diffusion has a known “moment duality” with respect to a K-dimensional pure-death
process related to a function of Kingman’s coalescent tree (see e.g. \cite{Etheridge09}, \cite{Griffiths10} and below, \ref{sec dual proc}). We will establish (Proposition \ref{prop sub dual}) that sWF processes also possess a moment dual pure-jump process in $\Z_+^K$, generalising the WF dual.  
We will prove that the new process is still non-increasing but may make jumps larger than one. The holding time distributions of the new dual coincide with the eigenvalues of the sWF transition operator, and at the same time also give the weights in its expansion as mixture of Dirichlet posteriors (see Section \ref{sec swf dual}), regulating borrowing strength. We will use the latter to generalise existing exact simulation methods for WF diffusions \cite{Jenkins_2017, JSant22} to encompass sWF distributions.\\
Duality has been first used for applications to filtering problems by \cite{Papaspiliopoulos_2014} where the model was framed as a Hidden Markov Model (HMM) and the dual process was shown to guarantee computability of the filtering distributions. We will generalise this here to sWF (Section \ref{sec posterior}). For Markovian sWF (subordinator clocks), the HMM structure is immediately preserved; for non-Markovian sWF signals (inverse subordinator clocks), we will consider the filtering problem for the augmented {\em Markov} signal $V=(X\circ C, C):$ the model is thus embedded in an HMM where the likelihood of the data depends on V only through its first coordinate $X\circ C.$ Posterior inference for $V$ and for both its coordinates $X\circ C $ and $C$ will be studied in Section \ref{sec frac posterior}. We will show that computability of the filter for $X\circ C$ is again guaranteed by the dual, as long as the {\em prior} distribution of $C$ is amenable to exact sampling. The sWF provide, for the first time, instances of computable posteriors (or filters) for time-reversible Dirichlet signals, induced by a dual process which does not obey one condition given in \cite{Papaspiliopoulos_2014}, and was never violated in any of the subsequent generalisations (\cite{omiros_2016, kkk,GRSZ21}): that the dual evolves almost surely through negative {\em unit} jumps. Indeed sWFs demonstrate that the pure-death property assumption on the dual is not needed to ensure computability. In fact, the main feature of the dual, responsible for computability would rather appear to be its almost-surely descending trajectories.
Unit vs non-unit negative jumps seem to just reflect the (dis)continuity properties of the sWF signal. \\
 All the above results hinge on the fact that duality is preserved under time-change (Proposition \ref{prop sub dual}), a very simple observation which, to the best of our knowledge, has not been stated explicitly, nor used for applications anywhere so far, but could be easily leveraged in many other contexts. Denoising in Machine Learning seems a plausible promising area. Still within the context of filtering time-dependent priors, we expect our approach to work also for the analysis of other non-Dirichlet classes of subordinated signals, for example those within the class of Pearson diffusions analysed in \cite{Papaspiliopoulos_2014} and \cite{kkk}, as well as their nonparametric, measure-valued extensions such as Fleming-Viot processes and Dawson-Watanabe superprocesses (time-dependent Gamma completely random measures)  (\cite{omiros_2016}).
\subsection{Structure of the paper}
The paper will begin with a recollection of the main distributional properties of the WF diffusion in Section \ref{sec wf properties}, including the moment duality relation between WF and the death process of Kingman's \lq\lq typed\rq\rq\ coalescent. In Section \ref{sec time change}, sWF processes will be defined; we will first introduce the two types of clock used for subordination; then we will describe in detail the distributional properties of sWF (generator, transition function, resolvent operator, memory, evolution equation); duality will be finally established between any sWF and a non-decreasing discrete process, extending the Coalescent WF-dual. In Section \ref{sec algorithms}, exact sampling schemes for the sWF transition function are established and illustrated. In Section \ref{sec posterior}, the posterior distributions for a sWF prior are explicitly derived and computable inference of the corresponding filter is discussed, where we distinguish between the Markovian and non-Markovian cases. The proofs of all the statements will be deferred to the Appendix.
\section{Wright Fisher diffusions}\label{sec wf properties}
In this Section we recall relevant distributional properties of the neutral $K$-type WF diffusion. The diffusion's dynamics are fully determined by a vector of (parent-independent) mutation parameters $\theta=(\theta_1\ldots,\theta_K)\in\R_+^K $ and
the associated generator
\begin{align}
    \label{WF gen}
    {\cal L} &= \frac{1}{2} \sum_{i=1}^K{(\theta_i-\vert {\theta} \vert x_i)} \frac{\partial}{\partial x_i} + \frac{1}{2}\sum_{i,j=1}^{K}{x_i(\delta_{ij}-x_j)}\frac{\partial^2}{\partial x_i \partial x_j},
\end{align}where $\delta_{ij}$ is the Kronecker delta, i.e., $\delta_{ij}=1$ if and only if $i=j$ and $\delta_{ij}=0$ otherwise, and
$\vert {\theta}\vert = \sum_{i=1}^K{\theta_i}$. The domain of the generator is the space $C^2(\Delta_K)$ of functions with continuous derivatives up to the second order.
The operator ${\cal L}$ modulates the increments of the diffusion over any infinitesimal time-interval (i.e. ${\cal L} f( x):= \lim_{t\downarrow 0}t^{-1}\left[\E[f( X(t))\mid  X(0)= x]-f( x)\right]$). Continuity (diffusiveness) of the trajectories can be read off its elliptic form.
The associated transition function $p_t({{x}},\dif {u})=\mathbb{P}[X(t)\in \dif u\mid X_0= x] $    
admits a representation as mixture of posterior Dirichlet distributions (see e.g. \cite{griffiths_1979, TAVARE1984, GRIFFITHS_Li_1983, Griffiths10}), for which Algorithm \ref{alg:simulating_WF} provides an exact sampling method. Let $\mathcal{M}(\boldsymbol{l};\vert \boldsymbol{l}\vert,{{x}})=\frac{\vert \boldsymbol{l}\vert !}{l_1! \cdots l_K !}\prod_{i=1}^{K}{x_i^{l_i}}$   be the Multinomial distribution on $Z_+^K$ with parameter $(|\bs l|, x)$ and, for every $\theta\in(0,\infty)^K$, denote with ${\cal D}_{\theta}$ the Dirichlet distribution on $\Delta_K$, with Lebesque-density
$$\mathcal{D}_{{\theta}}( x)=\frac{\prod_{i=1}^{K}{\Gamma(\theta_i)}}{\Gamma(\vert{\theta}\vert)}\prod_{i=1}^{K}{x_i^{\theta_i-1}}\mathbb{I}_{\Delta_K}( x),$$
where  $\Gamma(a)=\int_0^\infty x^{a-1} e^{-x} \dif x$ for $a>0$.
Define $\lambda_{m}:=m(m + \vert {\theta}\vert -1)/2$, $m=0,1,\ldots$. 
\begin{algorithm}
\caption{Exact simulation of $p_t(x,\cdot)$ \cite{Jenkins_2017}.}
\label{alg:simulating_WF}
\begin{algorithmic}[1]
    \State Simulate $|Z(t)| \sim q_\bullet^{|{\theta}|}(t)$
    
    \State Given $|Z(t)| = m$, simulate $\bs L \sim \mathcal{M}(m, x)$ 
    
    \State Given $\bs L = (l_1, l_2, \dots, l_K)$, simulate $U \sim \mathcal{D}_{{\theta} + \boldsymbol{l}}$
    
    \State \Return $U$
\end{algorithmic}
\end{algorithm}
For any $t$, $q_{\bullet}^{|{\theta}|}(t)=\{q_n^{|{\theta}|}(t):n\in\Z_+\}$ is a probability mass function on $\Z_+$; namely, it describes the distribution of the state at time $t$ of a pure-death process $(|Z(t)|:t\geq 0
)$ on $\mathbb{Z}_+\cup\{\infty\}$, leaving infinitely fast its entrance state $\{\infty\}$ and jumping from each $n$ to $n-1$ at rate $\lambda_n$. Such a process  
has a well-established {\em stochastic duality} relation with the WF diffusion, which will be recalled in section \ref{sec dual proc} (see also  \cite{TAVARE1984, griffiths06}). 
The probability $q_\bullet^{\vert{\theta}\vert}(t)$, has an infinite-series expression, which we will generalise for sWFs in Proposition \ref{prop sub dual}, formula \eqref{eq wtq}. The only non-trivial step in Algorithm \ref{alg:simulating_WF} is thus the first one, sampling from $q_\bullet^{\vert{\theta}\vert}(t)$. An exact method is provided by Algorithm 2 in \cite{Jenkins_2017}, which will be extended into our Algorithm \ref{alg:simulating_sub_dual} to cover sWF duals.

From the structure of Algorithm \ref{alg:simulating_WF} and the observation
\begin{align*}
    \mathcal{D}_{{\theta}}(\dif {{x}})\mathcal{M}(\boldsymbol{l};k,{{x}})\mathcal{D}_{{\theta}+\boldsymbol{l}}(\dif {u})=\mathcal{D}_{{\theta}}(\dif {u})\mathcal{M}(\boldsymbol{l};k,{u})\mathcal{D}_{{\theta}+\boldsymbol{l}}(\dif {{x}}).
\end{align*}
follow stationarity, i.e. $\int_0^1{\mathcal{D}_{{\theta}}(\dif {{x}})p_t({{x}},{u})}=\mathcal{D}_{{\theta}}(\dif {u})$, and reversibility, i.e. detailed i.e. ${\mathcal{D}_{{\theta}}(\dif {{x}})p_t({{x}},{u})}={\mathcal{D}_{{\theta}}(\dif {u})p_t({u},{{x}})}$ for all $x,u\in\Delta_K, ~t\geq 0$.
\\
Two equivalent series representations are available for $p_t$, one as probabilistic mixture of posterior Dirichlet (yielding Algorithm \ref{alg:simulating_WF}), the other in terms of orthogonal polynomial kernels see \cite{Griffiths_2013}. Both representations have a sWF extension, which will be shown in Proposition \ref{prop tf sub}.
\subsection{Dual process}\label{sec dual proc}
Let $X$ and $Z$ two stochastic processes with state spaces $E$ and $F$, respectively 

We say that stochastic duality holds between $X$ and $Z$ with respect to the duality function $h:E\times F\to \R$ if, for every $t$,
\begin{align}
    \label{eqn stochastic duality}
   \E[f(X(t))\vert X(0)=x]=\E[f(Z(t))\vert Z(0)=z]~\ \forall x\in E,z\in F.
\end{align}
 Define\begin{align}
\label{eqn the h function}
g({{x}},\boldsymbol{m})&:=\frac{\mathcal{M}(\boldsymbol{l};k,{{x}})}{\E_{\mathcal{D}_{{\theta}}}[\mathcal{M}(\boldsymbol{l};k,{{X}})]}= \frac{{(\vert {\theta}\vert)}_{\vert \boldsymbol{m}\vert }}{\prod_{i=1}^K{{(\theta_i)}_{m_i}}}\prod_{i=1}^K{x_i^{m_i}}.
\end{align}

In Bayesian terms, the function $g$ can be interpreted as the Radon-Nykodym derivative of the posterior Dirichlet with respect to the prior, under the assumption of categorical or multinomial likelihood. The well-known conjugacy property of ${\cal D}_{{ \theta}}$ with respect to multinomial sampling thus can be read as $\mathcal{D}_{{\theta}+\boldsymbol{l}}({{x}})=g({{x}},\boldsymbol{l})\mathcal{D}_{{\theta}}({{x}})$. 

When ${{X}}\sim WF$, the function $g$ of the form \eqref{eqn the h function} establishes a duality relation between ${{X}}$ and a multi-type pure-death process $ Z$ on the state-space $\Z_+^K$, 
where ${Z}$ evolves through jumps
\begin{align}
\label{bcprates}
\boldsymbol{m}\to \boldsymbol{m}-\boldsymbol{e}_i\ \ \textit{with rate}\ \ \lambda_{|\boldsymbol{m}|} \frac{m_i}{\boldsymbol{m}}, i=1,\ldots,K,\end{align}
where $\lambda_n:=\frac{1}{2}n(n+|{\theta}|-1)$ with $x\in\Z_+,~\theta\in\R_+^K$ 
and $\boldsymbol{e}_i=(\delta_{ij})_{j=1,\ldots,K}$ is the $K$-dimensional base vector. It is possible to see this by applying the WF generator to $g$ \cite{Etheridge09}, \cite{Griffiths10}. The generator $\widehat L$ of the dual is obtained, for every $({{x}},\boldsymbol{m})\in\Delta_K\times \Z_+^K$, from \begin{align}
\label{proof WF dual}
    {\cal L} g(\cdot,\boldsymbol{m})({{x}}) &= \lambda_{\totalm}\sum_{i=1}^K{\frac{m_i}{\totalm}[g({{x}},\boldsymbol{m}-\boldsymbol{e}_i)-g({{x}},\boldsymbol{m})]} =: \widehat{{\cal L}} g({{x}},\cdot) (\boldsymbol{m}).
\end{align} 
In population genetics, the dual process tracks the number of non-mutant surviving lines of descent of each type observed in a sample of $|\boldsymbol{m}|$ individuals.
 \begin{rmk}
The form \eqref{bcprates}-\eqref{proof WF dual} shows that 
    the process $|Z|$, defined by $|Z(t)|:=\sum_{i=1}^K Z_i(t),\ t\geq 0,$ is an autonomous pure death process in $\Z_+$ with transition rates from $n$ to $n-1$ equal to $\lambda_n$, precisely the process which, if started at $\{\infty\}$, has $q_{\bullet}(t)$ of Algorithm \ref{alg:simulating_WF} as time-marginal probabilities. 
 \end{rmk}
\section{Subordinated Wright-Fisher priors}\label{sec time change}

We will now proceed with the construction of the new class of time-dependent Dirichlet priors. 
\begin{defn}
Let ${X}=({X}(t):t\geq 0)$ be a WF$(\theta,\nu^*)$ diffusion process and $C=(C(t):t\geq 0)$ a non-decreasing Feller process with values in $[0,\infty]$, starting at $0$, independent of ${X}$. Define $\widetilde{ {X}}={X}\circ C$ via $\widetilde{ {X}}(t)= {X}(C(t))$. We call the process $\widetilde {{X}}$ a \emph{subordinated Wright-Fisher (sWF) process with random clock $C$, sWF$(\theta, C, \nu^*)$}. 
\end{defn}
The original WF($\theta,\nu^*$) diffusion is the special case where $C(t)\equiv t$ i.e. WF($\theta,\nu^*$)=sWF($\theta,{\rm Id}, \nu^*$). We will consider two large families of subordinated WF priors: (i) WF time-changed by a {\em L\'evy subordinator}; (ii) WF time-changed by the {\em inverse process} of a subordinator. Properties of sWF priors induced by compositions of type (ii) with (i) will follow directly.
Type (i) is the sole type of time-change that preserves the Markov property and introduce jumps. 
We will begin by describing the two types of clock (Section \ref{sec sub time change} and \ref{sec sub inv time change}). Then we will continue with a description of the transition function and memory of sWF (Section \ref{sec time-marginal}). Duality will be established in Section \ref{sec swf dual}. The sWF resolvent is derived explicitly in the Appendix.
\subsection{Time-change by subordinators}\label{sec sub time change}

A subordinator is an almost-surely non-decreasing process $S=(S(t):t\geq 0)$ starting at $0$, with stationary and independent increments. For a comprehensive treatment of subordinators see e.g. Chapter 6 in \cite{kenlevy}. 
Subordinators are a subset of the class of L\'evy processes, so their Laplace transforms admit a known  L\'evy-Khintchine representation (see e.g. \cite{kenlevy}, Theorem 30.1): for every $t$, $\E[e^{-\lambda S(t)}]=e^{-t\psi(\lambda)},$
where $\psi$ is in the form, for all $\lambda\geq 0$,\begin{align}
    \label{eqn sub laplace exponent}
    \psi(\lambda)&=\beta\lambda +\int_{(0,\infty)}{(1-e^{-\lambda s})\pi(\dif s)}
\end{align} where $\beta\geq 0$ is known as the drift and $\pi$ the L\'evy measure of $S$, satisfying 
\begin{align}
    \label{eq:sublevy}
\pi[(-\infty,0)]=0,
\ \ \textit{and}
\ \ \int_{0}^{\infty}{(s \wedge 1 )\pi(\dif s)}<\infty.
\end{align}
Notable examples of subordinators include:

\begin{itemize}
\item[(a)]
Poisson process: $\pi=c\dot \delta_1(\cdot)$ and $\psi(\lambda)=c(1-e^{-\lambda})$.
\item[(b)]
$\alpha$-stable process: $\pi(\dif s)= s^{-1-\alpha}\dif s$ and $\psi(\lambda) = \lambda^\alpha$ for $\alpha\in (0,1)$.
\item[(c)]
Gamma process: $\pi(\dif s)= a s^{-1}e^{-bs}\dif s$ and  $\psi(\lambda) = a\log(1+\frac{\lambda}{b})$ for $a,b>0$.\item[(d)] Inverse-Gaussian process: $\pi(\dif s)= (2\pi s^3)^{-{1/2}}\exp\{-\gamma^2 s/2\}$ and
 $\psi(\lambda)=\delta (\sqrt{2\lambda+\gamma ^2}-\gamma)$, for $\delta,\gamma\geq 0$. See Example 1.3.21 in \cite{applebaum2009levy} and \cite{Feller1971}.
\end{itemize}
If ${\cal L}$ is the generator of a WF process $X$ with transition function $p_t(x,\cdot)$, then the L\'evy pair $(\beta,\pi)$ determines the generator $\wt{\cal L}$ for the sWF process $\widetilde{ {X}}= {X}\circ S$ as 
 \begin{align}
  \label{eqn sWF gen 2}
    \widetilde{{\cal}L}f({x})=\beta {\cal L} f({x}) + \int_{(0,\infty)}{\big(\E[f({X}(s)\vert {X}(0)={x}] - f({x})\big)\pi(\dif s)},
\end{align} for any $f$ in the domain of $L$,
(see Theorem 32.1 and Remark 32.4 in \cite{kenlevy} for details).
The integral term is responsible for the occurrence of jumps in the trajectories of $\widetilde{{X}}$, whose intensity and sizes are governed by the L\'evy measure $\pi$. The drift $\beta$, on the other hand, scales the speed of the diffusive component ${\cal L}$ of the law of $\wt{X}$.
From \eqref{eqn sWF gen 2} it is possible to see that the degenerate case $(\beta,0)$ corresponds to an ordinary WF diffusion after a deterministic, linear time-change $t\mapsto \beta t$, the only instance of time-change which preserves, at the same time, both continuity and Markov property.

\subsection{Time-change by inverse subordinators}\label{sec sub inv time change}

Let $S=S(t)$ be a subordinator with characteristic pair $(\beta,\pi)$. The inverse process of $S$ is the process $R=(R(t):t\geq 0)$ defined by
$$R(t):=\inf\{u\geq 0: S(u)> t\}.$$
Under some regularity conditions on the L\'evy measure $\pi$ (for simplicity, we may assume $\pi(0,\infty)=\infty$), the distribution of $R(t)$ has a density given by
\begin{align*}
G_t(u)=\frac{d}{du}\int_{-\infty}^t \mu^s(u)\dif s,
\end{align*}
where $\mu^t:=\Prb(S(t)\in \dif u)$. Let $\bar\pi(s):=\int_s^\infty \pi(\dif u).$
The Dzerbayshan-Caputo derivative is defined by the convolutional integro-differential operator
\begin{align}
    \label{eq frac cond exp evo}
   I^{(\beta,\pi)}_t f(t):= \beta f{'}(t)+\int_0^t f{'}(t)\ \bar\pi( t-s)\dif s.
\end{align} 
generalised fractional dynamics are proposed in \cite{Kol09}. A modern treatment is in \cite{Toaldo2015}. Let $\widetilde{ {X}}={{X}}\circ R$ be a sWF$(|{\theta}|)$ prior with clock $R$. The non-Markovian process $\widetilde{ {X}}$ has well-understood dynamics:  the conditional expectation operator $\widetilde T_tf({x}):=\mathbb{E}[f(\widetilde{{X}}(t))|\widetilde{X}(0)={x}]$ is the solution to:
\begin{align}
 \label{gencaputo}
&I^{(\beta,\pi)}_t u(t) ={\cal L}\ u(t ), \ \ \ \  u(0)=f,\end{align}
where ${\cal L}$ is the Wright-Fisher diffusion generator \eqref{WF gen}.
The WF diffusion is re-obtained as the extreme case when $\pi(0,\infty)=0$ (pure constant drift), in which case $I^{\beta, 0}_t=\beta \dif/\dif t$ and \eqref{gencaputo} becomes a classical backwards Kolmogorov equation. A notable non-Markovian case is given when $\beta=0$ and $R$ is the inverse process of an $\alpha$-stable subordinator with Laplace exponent $\psi(\lambda)=\lambda^\alpha$, for $\alpha\in(0,1)$. Then $\bar\pi(s)=s^{-\alpha}/\Gamma(1-\alpha)$ and the operator $I^{(\beta,\pi)}$ reduces to the Caputo fractional derivative (\cite{Saichev98}, Appendix A)
\begin{align}
    I^{(0,\alpha)}_tf(t)=\frac{1}{\Gamma(1-\alpha)}\int_0^t\frac{f^{'}(s)}{(t-s)^{\alpha}}\dif s =\frac{\dif^{\alpha}}{\dif t^\alpha} f(t).
\end{align}
In this case, \eqref{gencaputo} describes a sub-diffusive, time-fractional (in the Caputo sense) dynamics (see \cite{meer2012, Kol09, LEONENKO2013}). 

\subsection{Distributional properties}\label{sec time-marginal}
We are now ready to derive key properties of subordinated WF priors. Let $\widetilde{ {X}}=(\widetilde {{X}}(t):t\geq 0)\sim$\ sWF$(\theta, C,\nu^*)$. For any $t\geq 0$ denote with $G_t$ the law of $C(t)$ and with $p_t({x},\dif {u})$ the transition function of the WF diffusion (corresponding to Algorithm \ref{alg:simulating_WF}). 
Then $\widetilde{{X}}(0)={X}(0)$ almost surely and, for every $t$, the conditional distribution of $\widetilde{{X}}(t)$, given its starting state $\widetilde{{X}}(0)= x$, 
can be derived from
\begin{align}
\label{eq:sub_sg}
\widetilde T_t f({x})=\int_0^\infty \E[f(  X(u))\mid  X(0)= x]G_t(\dif u)=\int_0^\infty\int_{\Delta_K} f(w)p_u( x,\dif w)G_t(\dif u),
\end{align}
for any bounded measurable $f$, so that
\begin{align}
    \label{thm general sub transition function}
  \widetilde{p}_t( x,\dif w)=\E[p_{C(t)}(x,\dif w)]~\forall~ x,w\in\Delta_K,~t\geq 0,
\end{align}
where the form on the right-hand side follows from an exchange of order of integration, allowed by the smoothness and Lebesgue-absolute continuity of ${p}_t$, the Feller property and the independence of $X$ and $C$. 
\subsubsection{Time-marginal distributions}
The first consequence of \eqref{thm general sub transition function} is that sWF processes are indeed time-dependent, time-reversible Dirichlet priors, by construction. The proof of the following is straightforward (apply Fubini to \eqref{eq:sub_sg}, using independence between $C$ and $\widetilde{{X}}$, and the result follows from stationarity/reversibility of ${\cal D}_{|{\theta}|}$ for the WF diffusion ${X}$).

\begin{prop}
  $\widetilde{{X}}$ is Dirichlet-stationary and Dirichlet-reversible.
\end{prop}

The second important consequence of \eqref{thm general sub transition function} is that the conditional distributions of a sWF model admit expansions very much similar to those of WF diffusions. Consider the Laplace transform of the clock process at time $t$, i.e.
$\Phi_t(\lambda)=\mathbb{E}[e^{-\lambda C(t)}].$
 We will use the notation $(a)_n:=\Gamma(a+n)/\Gamma(n)$ for $a>0,n\in\Z_+$ is the Pochammer symbol for generalised raising factorials. 
Define 
$Q_n({x},{u}):=\sum_{|\boldsymbol n|=n} Q_{\boldsymbol{n}}({x})Q_{\boldsymbol{n}}({u}),$ where $Q_{\boldsymbol{n}}$ are multivariate polynomial, 
orthogonal with respect to the weight measure $\mathcal{D}_{{\theta}}$, i.e. \cite{Griffiths_2013}
 $$\int_{\Delta_K} Q_n({x},{u})Q_m({u},{w}) \mathcal{D}_{{\theta}}(\dif {u})=Q_n(x,w)\delta_{{n}{m}},~\forall  x,  w\in\Delta_K,\ n,m\in\Z_+.$$
 Recall the notation $\lambda_n=\frac{n(n+|\theta|-1)}{2},~n=0,1,\ldots.$
\begin{prop}
\label{prop tf sub} For any given time $t\geq 0$ and ${{x}}\in \Delta_K$, the conditional distribution  $\widetilde{{X}}(t),$ given $\widetilde{{X}}(0)={{x}}$ admits the equivalent representations
\begin{align}
     \label{eqn sub tf polydual}
    \widetilde{p}_t({{x}},\dif {u})=\sum_{n=0}^\infty \Phi_t(\lambda_n)Q_n({x},{u})\mathcal{D}_{{\theta}}(\dif u)
\end{align}
and
\begin{align}
    \label{eqn tf subdual}
    \widetilde{p}_t({{x}},\dif u)=
    \sum_{m=0}^{\infty}{\widetilde{q}_m^{\vert{\theta}\vert}(t)\sum_{\vert \boldsymbol{l}\vert\leq m}{\mathcal{M}(\boldsymbol{l};m,{{x}})\mathcal{D}_{{\theta}+\boldsymbol{l}}(\dif u)}},
\end{align}
where 
\begin{align}
    \label{sub_bcp} \widetilde{q}_m^{|{\theta}|}(t):=\sum_{j=m}^{\infty}(-1)^{j-m}\Phi_t(\lambda_j)a_{j,m}^{(|{\theta}|)}\geq 0,\ \ \ m\in\Z_+,
\end{align}
and $$a_{j,m}^{(|{\theta}|)}:=\frac{(2j+|{\theta}|-1)(m+ |{\theta}|)_{(j-1)}}{m!(j-m)!}.$$ \end{prop}
\begin{rmk}(i)  The known expansions for the  ordinary WF transition function \cite{griffiths06,Griffiths10}, are re-obtained as the special case 
$\Phi_{t}(\lambda_n)=e^{-\lambda_n t}.$
In this case, the coefficients $\wt q_m(t)$ in \eqref{sub_bcp} give the form of the probabilities $q_\bullet(t)$ appearing in Algorithm \ref{alg:simulating_WF}. The proof for general sWF follows simply by replacing $t$ (deterministic) with $C(t)$ (random) and taking the expectation. \\
(ii) For $K=2$ (Beta stationary), if $C$ chosen to be the inverse of an $\alpha$-stable subordinator, the process $\widetilde{X}\circ C$ has Mittag-Leffler eigenvalues $E_\alpha(-\lambda_n t^{\alpha})$ as in \eqref{mittagleffler} and $\wt p_t$ becomes the transition function of (a shifted version of) the so-called fractional Jacobi diffusion, a member of the Fractional Pearson family which has been studied by Leonenko \cite{LEONENKO2013}. Thus, modulo a change of variable $J(t):=2\widetilde X(t)-1$, formula \eqref{eqn tf subdual} gives a novel alternative representation of fractional Jacobi diffusions. Furthermore, for general $K>2$, sWF provide a $K$-dimensional generalisation of Leonenko's fractional Jacobi diffusions, which, to our knowledge, is new.
\end{rmk}

\subsubsection{Resolvent}
\label{sec swf resolvent}
The finite-dimensional distributions $\wt p_t$ of a sWF process 
$\widetilde {{X}}={X}\circ C$ can be characterised via the corresponding {\em resolvent measure}: 
    $$
   {\cal R}_\lambda({x},\dif u)=  \int_0^\infty e^{-\lambda t } \widetilde{p}_t({{x}},\dif u)\ \dif t.
    $$
The proof of the next Corolloary follows by Proposition \ref{prop tf sub}, and by exchanging integration with summation, which is allowed by the smoothness of $p_t,$ $Q_n$ and $e^{\lambda_n t}$.
\begin{cor}
    The resolvent measure of a sWF$({\theta})$ process with clock $C$ takes the two equivalent forms
\begin{equation} \label{eqn res m1}
     {\cal R}_\lambda({x},\dif u) =
      \sum_{n=0}^\infty \Phi(\lambda, \lambda_n)
      Q_n({x},u)
      {\cal D}_{{\theta}}(\dif u),
     \end{equation}
    and 
    \begin{align}
    \label{eqn res m2}
    {\cal R}_\lambda({x},\dif u) &=\sum_{m=0}^\infty \chi_m(\lambda)\ \sum_{\vert \boldsymbol{l}\vert\leq m}{\mathcal{M}(\boldsymbol{l};m,{{x}})\mathcal{D}_{{\theta}+\boldsymbol{l}}(\dif u)}, 
 \end{align}
   where ${\Phi}(\lambda, \lambda_n)=\mathbb{E}\left[e^{-\lambda t}\Phi_t(\lambda_n)\right]$ and
   \begin{align}
       \label{eqn sub dual resolvent}
\chi_m(\lambda)&:=\sum_{j=m}^{\infty} {\Phi}_n(\lambda,\lambda_j)a_{j,m}^{|{\theta}|}.\end{align}
\end{cor}
In particular, if $C=S$, a $\psi$- subordinator, then \begin{align}\Psi_n(\lambda,\lambda_n)=\int_0^\infty e^{-\lambda t}e^{-\psi(\lambda_n)t}\dif t=\frac{1}{\lambda+\psi(\lambda_n)};\end{align}
 if $C=R$, the inverse of $S$, then (See
 \eqref{eqn doulelap}):
\begin{align}
\Phi(\lambda,\lambda_n)=\frac{\psi(\lambda)}{\lambda(\lambda_n+\psi(\lambda))}.
    \end{align}
    
\subsubsection{Memory properties}\label{sec memory}

The type of time-memory of a sWF process is described entirely by the time-dependence of the Laplace transform $t\mapsto \Phi_t$ of the clock $C$. Indeed, the expansion \eqref{eqn sub tf polydual} shows that $\Phi$ determines all the eigenvalues of the conditional expectation operator $(\widetilde T_t f)$: Consider, for any $u\in\Delta_K,$ the orthogonal functions ${x}\mapsto Q_k({x},u)$. Indeed, by orthogonality and from \eqref{thm general sub transition function},
    \begin{align}
    \label{fracsg}
       \widetilde{T}_t Q_k(x,u)&=\sum_n\Phi_t(\lambda_n)\int_{\Delta_k}Q_n(x,w)Q_k(w,u){\cal D}_{{\theta}}(\dif w)\notag\\&=\Phi_t(\lambda_k)Q_k(x,u).
    \end{align}
    In fact, for any $f\in L^2(\Delta_K,{\cal D}_{{\theta}}),$ we can write  
\begin{align*}
    \widetilde{T}_t f(x)=\sum_n\Phi_t(\lambda_n)\widehat{ f}_n({x}),\ \ \ \ \ \ \text{where} \ \ \widehat f_n({x}):=\int_{\Delta_k}f(u)Q_n(x,u){\cal D}_{{\theta}}(\dif u).\end{align*}
     The function $\Phi_t(\lambda)$ is a completely monotone function with $\Phi_t(0)=\Phi_0(\lambda)=1,$ in particular it is decreasing in $\lambda$. Therefore, the largest non-trivial eigenvalue is $\Phi_t(\lambda_1)=\Phi_t(|\theta|)$, which represents the rate of convergence to the stationary distribution $\mathcal{D}_{{\theta}}({x})$. 
       If $C=S$ for a subordinator $S$ with Laplace exponent $\psi=\psi(\beta,\pi)$, the eigenvalues of $\wt{T}_t$ will be, for every $n,$ of the form $\Phi_t(\lambda_n)=e^{-t\psi(\lambda_n)}$.
      If $C=R$ (inverse of $S$), the memory becomes heavy-tailed and 
      $\Phi_t$ can be described as the solution of a generalised Volterra-type kinetic equation as follows.

\begin{prop} 
\label{prop frac conv eqn}
    If $R$ is the inverse of a $\psi(\beta,\pi)$-subordinator, the Laplace transform Laplace transform $\Phi_t(\lambda)=\mathbb{E}[e^{-\lambda R(t)}],$ as a function of $t$, is solution to the equation: for every $\lambda\geq 0,$
    \begin{align}
    \label{eqn gen kin}
    \Phi_t(\lambda)-\Phi_0(\lambda)=-\lambda\int_0^t \kappa(t-s)\Phi_s(\lambda) \dif s,\ \ \ \textit{and}\ \ \ \Phi_0(\lambda)=1,
    \end{align}
    where $\kappa(t):=\int_0^\infty \mu^s(t)\dif s$ is the {\em potential density} of the underlying $\psi$ subordinator.
    \\
    Equivalently $t\mapsto\Phi_t$ solves, for any $\lambda$, 
\begin{align}
\label{eqn phi fractional}
    I^{(\beta,\pi)}_t\Phi_t(\lambda)=-\lambda\Phi_t(\lambda),\ \ \phi_0(\lambda)=1.
\end{align}
   where $I^{(\beta,\pi)}_t$ is defined as in \eqref{eq frac cond exp evo}. \end{prop}
   
 \begin{rmk}
 \begin{itemize}
     \item[(a)]
     The exponential memory is recovered from the kinetic equation \eqref{eqn gen kin} when $k(t)\equiv 1$ and $I^{(\beta,\pi)}_t$ reduces to ordinary derivative. This solution, applied to $\lambda_n$, gives the eigenvalues of the ordinary WF semigroup; if applied to $\widetilde{\lambda}_n:={\psi}^*(\lambda_n),$ it gives the eigenvalues for $\widetilde{{X}}={X}\circ S$ for $S$ a $\psi^*$-subordinator, independent of $R$.
     \item[(b)] If $R$ is the inverse of an $\alpha$- stable subordinator with $\alpha\in(0,1)$
 the Laplace transform solving \eqref{eqn gen kin} is the Mittag-Leffler function
\begin{align}
\label{mittagleffler}
\Phi^\alpha_t(\lambda)=\sum_{n=0}^\infty\frac{(-\lambda t^\alpha)^n}{\Gamma (\alpha n+1)}=:E_{\alpha}(-\lambda t^\alpha).
\end{align}
The potential measure in this case is $\kappa_\alpha(t)=t^{\alpha-1}/\Gamma(\alpha).$
Notice that when $\alpha\to 1$, $\kappa_\alpha(t)\to 1$ and the equation \eqref{eqn gen kin} reduces back to the Markovian case.
\item[(c)]
Let $S_1$ and $S_2$ be two subordinators with Laplace exponents $\psi_1,\psi_2$, respectively, and let $R_1,R_2$ be their corresponding inverse processes. 
Then the evolution of the sWF process $\widetilde{{X}}^*={X}\circ C$ with clock $C=R_2\circ S_1$, i.e. $C(t)=R_2(S_1(t)),$ has the same structure as \eqref{eqn tf subdual}-\eqref{eqn sub tf polydual} with eigenvalues
\begin{align}\Phi_t^*(\lambda_n)= \E\left[\Phi_{1,R(t)}(\lambda_n)\right]=\Phi_{2,t}(\psi_{1}(\lambda_n)),
\label{eq mix sub}
\end{align}where $\Phi_{2,t}(\lambda)$  
solves the convolutional equation \eqref{eqn gen kin} for $\psi_2$. 
\end{itemize}
 \end{rmk}
\begin{figure}[H]
\centering
\includegraphics[width=\linewidth]{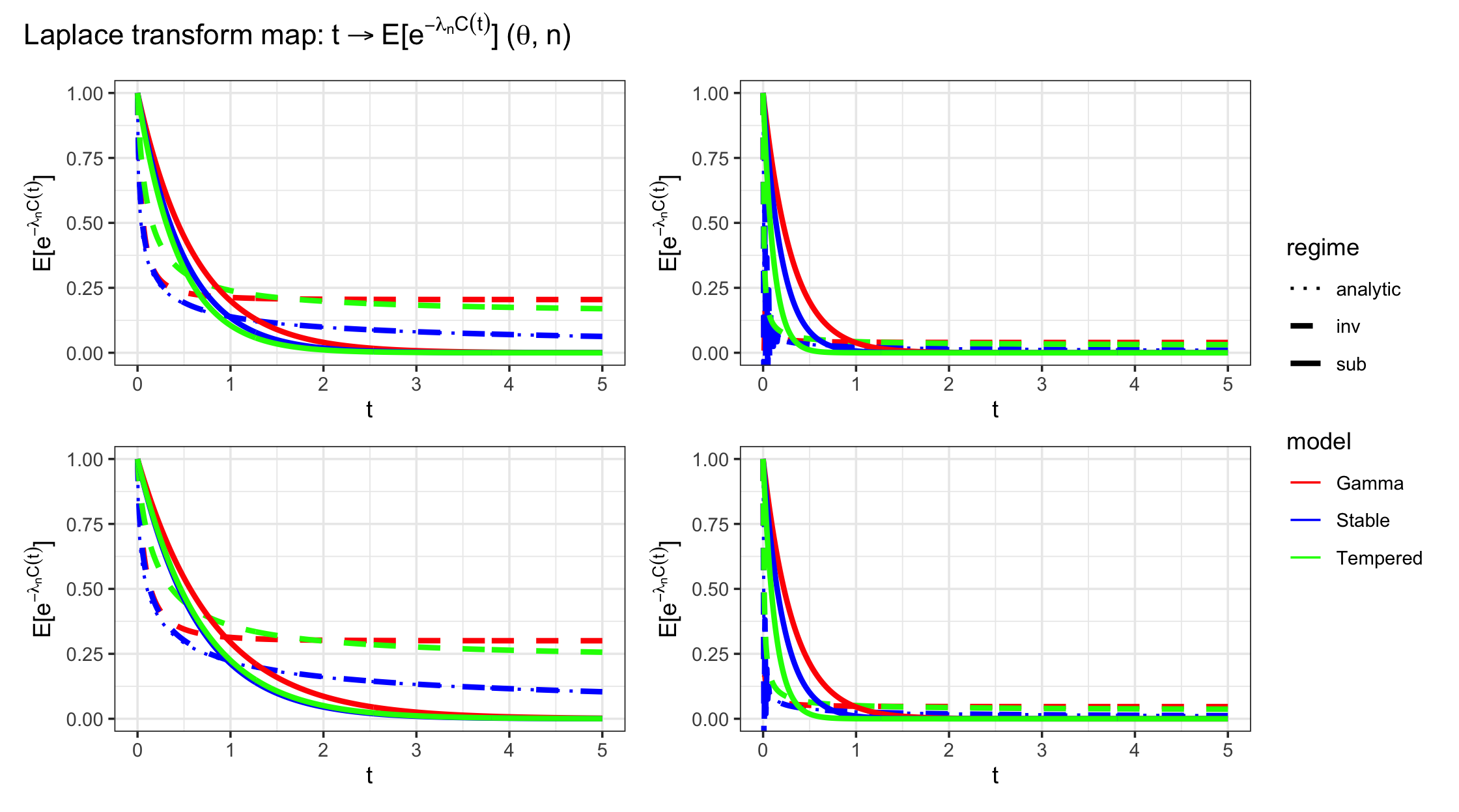}
\caption{
Laplace transform maps $t \mapsto \mathbb{E}[e^{- \lambda_n C(t)}]$ for $\alpha$-stable ($\alpha=0.5$), Gamma process ($b=1$), and tempered-stable ($\alpha=0.7$, $\lambda=0.5$), subordinators, and their inverse subordinator counterparts, computed for, clockwise from the topleft: $(\theta,n)=(1,2),(1,5),(0.2,2),(0.2,5)$. The solid lines show the analytical Laplace transforms of the subordinators, the dashed lines the numerically-computed Laplace transforms of the corresponding inverse subordinators and, for the $\alpha$-stable inverse subordinator, the dotted line corresponds to the analytical Mittag-Leffler expression $E_{\alpha}(-\lambda_n t^\alpha)$.
}

  \label{fig: Fractional Memory Comparison}
\end{figure}

Figure \ref{fig: Fractional Memory Comparison} illustrates the difference between subordinator and subordinator-inverse clocks in terms of the decay with $t$ of the sWF process's second largest eigenvalue, exponential for subordinator clocks and heavy-tailed for inverse subordinator clocks, inducing a slower convergence to $0$ in $t$ . 

Figure A-1
in the Appendix shows a wider choice of subordinator clocks and how the sWF eigenmodes decay in terms of $t,\theta$ and $n$. When considering \eqref{eqn tf subdual}, this slower decay translates to relatively higher weights on Dirichlet-Multinomial components with higher total mass $m$, when compared to the transition function of the WF diffusion, strengthening the time-correlation.

\subsection{Duality}\label{sec swf dual}
As mentioned, much of the time-dependence structure, and computability, of Wright-Fisher diffusions can be read off its dual process. A simple by key observation is that changing the time with an independent clock preserved stochastic duality.

\begin{prop}\label{prop sub dual}
Let $X$ and $Z$ be (Feller) Markov processes with state space $E$ and $F$ respectively. Assume that $X$ is dual to $Z$ with respect to some duality function $h$. Define $\widetilde{X}=X\circ C$ to be the time-change of $X$ via $C$. Then $\widetilde{X}$ is dual to $\widetilde{Z}:=Z\circ C$ with respect to the same duality function $h$.
\end{prop}
\begin{rmk}\label{rmk augdual}
    Proposition \ref{prop sub dual} establishes the dual equation 
    \eqref{eqn stochastic duality} for every $t$ so, when the Chapman-Kolmogorov property fails, in general \eqref{eqn stochastic duality}  does not necessarily characterize the full law of the dual. On the other hand, if $C$ is an inverse subordinator with dynamics
    described by \eqref{eq frac cond exp evo}, then \eqref{eqn stochastic duality} also implies duality between the respective evolution equations i.e.:
    $$ I^{(\beta,\pi)}_t\E_x[g(\wt X(t);z)]={\cal G}g(\cdot, z)(x)={\cal A}g(x,\cdot)(z)=I^{(\beta,\pi)}_t\widehat\E_z[g(x;\wt Z(t))],$$
    where ${\cal G}$ and ${\cal A}$ are, respectively, the generators of $X$ and $Z$. 
    The same holds when $C=R\circ S$ where $S$ is a subordinator and $R$ an inverse subordinator: ${\cal G},{\cal A}$ will be replaced by the generators of $X\circ S$ and $Z\circ S$.
    \end{rmk}

\begin{prop}\label{prop t-dual}
Let $\widetilde X$ be a sWF$(\bs{\theta, C,\nu^*})$ and denote with $\widetilde Z$ its dual with respect to the duality function $g=g({x},\boldsymbol{n})$ as defined in \eqref{eqn stochastic duality}. Then $\widetilde{Z}$ is a non-increasing pure-jump process in $\Z_+^K,$ with time-marginal distributions
\begin{align}
    \label{eq wtq}
{\widetilde{q}}^{ *}_{\boldsymbol{m}, \boldsymbol{l}}(t) :=\mathbb{P}[\widetilde{Z}(t)=\boldsymbol{l}\mid \widetilde{Z}(0)=\boldsymbol{m}]= \widetilde{q}_{|\boldsymbol{m}|, |\boldsymbol{l}|}(t)H(\boldsymbol l,\boldsymbol m),
\end{align}
where, for any $0\leq k\leq m<\infty$, the coefficients $\widetilde{q}_{m,k}$ are defined by 
\begin{align}
 \label{eqn:dual_from_m}
 \wt{q}^*_{\boldsymbol{m},\boldsymbol{l}}(t) =\wt{q}_{|\boldsymbol{m}|,|\boldsymbol{m}-\boldsymbol{l}|}(t) H(\boldsymbol{l}; \boldsymbol{m} , |\boldsymbol{l}|),
 \end{align}
 where $\wt{q}_{m,k}(t)$ has the same form as $\wt q_k(t)$ of formula \eqref{sub_bcp}, with weights $a_{j,k}^{(|{\theta}|)}$ replaced by ${a}_{j,k}^{(|\theta|,m)}:=a^{(|{\theta}|)}_{j,k}\frac{m_{[j]}}{(m+\theta)_{(j)}}.$
 Here $a_{[x]}$ denotes the (generalised) descending factorial i.e. $a_{[x]}=\Gamma(a+1)/\Gamma(a-x+1)$, and $H$ denotes the hypergeometric distribution.
The probability that $\wt Z$ does not leave its starting state by time $t$ is 
\begin{align}
\label{eq wtq surv}
   \widetilde q^*_{\boldsymbol{n},\boldsymbol{n}}(t)= \mathbb{P}[Z(t)=\boldsymbol{n}\mid Z(0)=\boldsymbol{n}]=\Phi_t(\lambda_{|\boldsymbol{n}|}).
\end{align}
In particular, \\
(i) if $C=S$ for a subordinator $S$ with L\'evy exponent $\psi(\beta,\pi)$, then $\widetilde Z$ is a pure jump Markov process in $\mathbb Z_+^K$ with negative jumps
\begin{align}
\label{eqn subrates}
    \boldsymbol{n}\to \boldsymbol{k}\ \ \textit{with rate}\  \widetilde{\lambda}_{|\boldsymbol{n}|, |\boldsymbol{k}|} H(\boldsymbol k,\boldsymbol n),\ \ \ \bs k\leq \bs n,
\end{align}
where
\begin{align}
\label{eqn subrates 1}
   \widetilde{\lambda}_{|\boldsymbol{n}|, |\boldsymbol{k}|}=\sum_{j=|\boldsymbol{k}|}^{|\boldsymbol{n}|}  (-1)^{j-|\boldsymbol{k}|+1}a^{(|{\theta}|,|\boldsymbol{n}|)}_{j,|\boldsymbol{k}|}\widetilde\lambda_j,
\end{align}
and \begin{align}
\label{eq:wtlambda}
   \widetilde{\lambda}_n:=\beta \lambda_n+\int_{0^+}^\infty\left(1-e^{-u\lambda_n}\right)\ \pi(\dif u),\ \ n\in\Z_+.
\end{align}\\
(ii) If $C=R$, where $R$ is the inverse of the subordinator $S$, then $\widetilde Z$ is a non-increasing process, with evolution described by 
\begin{align}
 \label{eq wtq inv}
    \widetilde{q}^{ *}_{\boldsymbol{n}, \boldsymbol{k}}(t)=\delta_{\boldsymbol{n}, \boldsymbol{k}}+\left(\sum_{j=|\boldsymbol{k}|}^{|\boldsymbol{n}|} (-1)^{j-|\boldsymbol{k}|+1}a^{(|{\theta}|,|\boldsymbol{n}|)}_{j,|\boldsymbol{k}|}\ \lambda_j\int_0^t\Phi_s(\lambda_j) \kappa(t-s)\dif s\right)H(\boldsymbol k,\boldsymbol n)
\end{align}
where $\kappa$ is the potential density of $S$. Equivalently,
\begin{align}
\label{eqn dual frac rates}
I^{(\beta,\pi)}_t{\widetilde{q}}^{ *}_{\boldsymbol{n}, \boldsymbol{k}}(t)=\left(\sum_{j=|\boldsymbol{k}|}^{|\boldsymbol{n}|} (-1)^{j-|\boldsymbol{k}|+1}a^{(|{\theta}|,|\boldsymbol{n}|)}_{j,|\boldsymbol{k}|}\ \lambda_j\Phi_t(\lambda_j) \right)H(\boldsymbol k,\boldsymbol n),\end{align}
where $I_t^{(\beta,\pi)}$ is the generalised fractional Caputo derivative from \eqref{eq frac cond exp evo}, so that 
\begin{align}
\label{fracrates}
I^{(\beta,\pi)}_t{\widetilde{q}}^{ *}_{\boldsymbol{n}, \boldsymbol{k}}(t)\bigg\vert_{t=0}=\lambda_{|\boldsymbol{n}|}\frac{n_i}{|\boldsymbol{n}|}\delta_{\boldsymbol{k}, \boldsymbol{n}-\boldsymbol{e}_i},\ \ \ \  i=1,\ldots K\end{align}
where $\boldsymbol{e}_i$ is the vector with the $i$-th coordinate equal to 1 and all other coordinates zero, and $\delta_{\cdot,\cdot}$ is the Kronecker delta. \\
\end{prop}

   Equation \eqref{fracrates} shows that, after waiting a (non-exponential) $\Phi$-distributed amount of time, the only jumps that can occur in the dual processe are almost surely of size one. Thus, while time-changing the WF dual by a subordinator $S$ induces a non-increasing Markov process that may make jumps larger than one, as shown by \eqref{eqn subrates 1}, time-changing it by an inverse subordinator induces a non-increasing, non-Markovian process evolving almost surely solely through unit jumps. Applying a mixed clock $C=R\circ S_1$, wiith $S_1$ an independent subordinator with L\'evy exponent $\psi_1(\lambda)$, then equation \eqref{fracrates} would become of the form
   \begin{align}
       \label{fracratesmix}
        I^{(\beta,\pi)}_t{\widetilde{q}}^{ *}_{\boldsymbol{n}, \boldsymbol{k}}(t)\bigg\vert_{t=0}
        =\sum_{j=|\boldsymbol{k}|}^{|\boldsymbol{n}|}  (-1)^{j-|\boldsymbol{k}|+1}a^{(|{\theta}|,|\boldsymbol{n}|)}_{j,|\boldsymbol{k}|}\psi_1(\lambda_j),\ \ \ \  i=1,\ldots K,
\end{align}
inheriting the shape \eqref{eqn subrates 1}, which shows that the dual in this mixed case may descend through non-unit jumps. Notice that both \eqref{eqn subrates 1} and \eqref{fracratesmix} vanish whenever $\bs k>\bs n$.
\begin{figure}[htp]
\centering
\includegraphics[width=\linewidth]{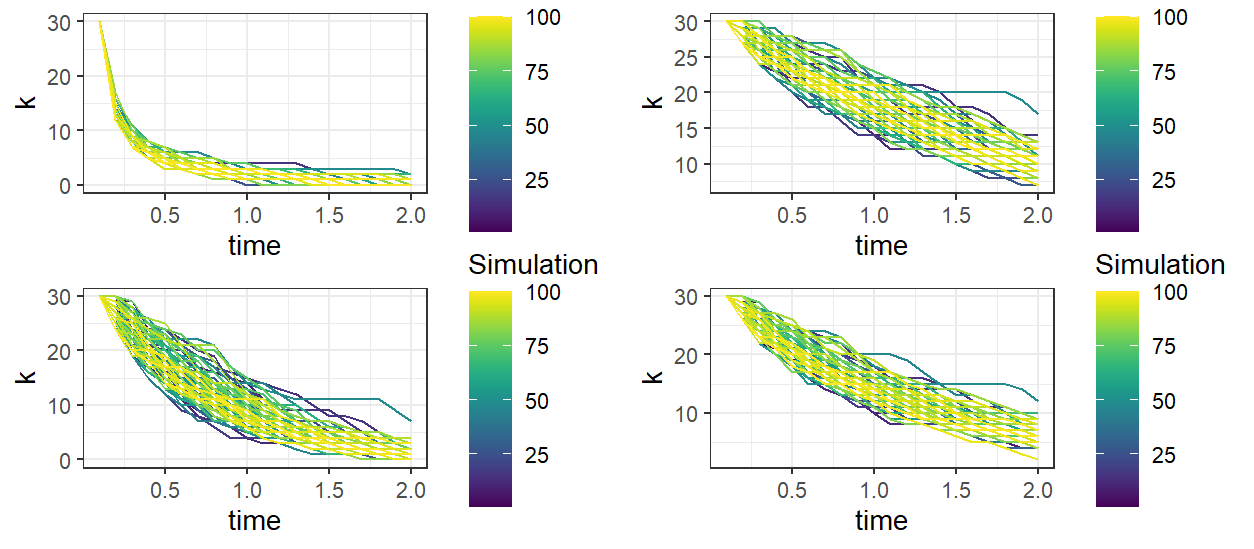}
  \caption{100 Simulated trajectories of the sWF's dual process for $S$ (a) the identity, (b) a $1/2$-stable subordinator, (c) an inverse-gaussian process and (d) an inverse-gamma process with shape $5$ and scale $2$. Each process shares a common mutation total of $|{\theta}|=2$ and drift of $\beta=0.01$. All duals were initiated at $n=30$ and sampled over $t\in[0,2]$ at regular intervals of $0.1$. Graphs are presented clockwise top left to bottom left.}
  \label{fig: sub qnk trajectories}
\end{figure}

\section{Exact simulation}\label{sec algorithms}
The expansion \eqref{eqn tf subdual} shows that, in order to sample from the distribution $\widetilde{p}_t(x,\cdot)$ of a sWF process, conditional on its initial state $x$, one has to sample from a mixture of posterior Dirichlet distributions, just like for WF diffusions. The only difference lies in the weights $\widetilde{q}_{k} (t)$ of the mixture. Such weights are given, for fixed $t$, by the distribution of the time-changed dual $|\widetilde Z(t)|$ and, just as in \cite{Jenkins_2017}, sampling exactly from these weights is the main difficulty. 
There are two main approaches, both leading to an extension of Jenkins and Span\`o's Algorithm \ref{alg:simulating_WF}.\\
One option (Option A in Algorithm \ref{alg:simulating_TCWF} below) is to first sample a random time, $C(t)$, and to then input this as the value of $t$ in Algorithm \ref{alg:simulating_WF}. The other option (Option B), when possible, is to sample exactly from $\widetilde{q}_{k} (t)$ directly, generalising an existing Algorithm 2 of \cite{Jenkins_2017}.

\begin{algorithm}
\caption{Exact simulation of $\widetilde{p}_t(x,\cdot)$}
\label{alg:simulating_TCWF}
\begin{algorithmic}[1] 
    \State Simulate $\widetilde{q}_k^{\vert{\theta}\vert}(t)$ using one of two options:
        \Option{Option A:}
        \Indent{Simulate $s\sim C(t)$.}
        \Indent{Simulate $m \sim q_{\cdot}^{|{\theta}|}(s)$ using Algorithm 2 from \cite{Jenkins_2017}.}
        \EndIndent
    \EndIndent
        \Option{Option B:}
        \Indent{Simulate $m \sim \widetilde{q}_{\cdot}^{\vert{\theta}\vert}(t)$ using Algorithm \ref{alg:simulating_sub_dual}.}
        \EndIndent
    \EndIndent
    
    \State Given $m$, simulate $\boldsymbol{l} \sim \mathcal{M}(m,x)$.
    \State Given $\boldsymbol{l}=(l_1,l_2,...,l_K)$, simulate $U \sim \mathcal{D}_{{\theta}+\boldsymbol{l}}$.
    \State \Return $U$.
\end{algorithmic}
\end{algorithm}

Option A is very natural. Its feasibility depends on the law of the clock, $C$, i.e. subordinator or inverse-subordinator: one can use Algorithm 4.3 from \cite{Dassios2019} for (suitable) subordinators and Algorithm 1 in \cite{Aleks2023} for inverse subordinators. Both algorithms however carry restrictions on the class of admissible clocks. For example, \cite{Aleks2023} provides algorithms for a large class of inverse subordinators with L\'{e}vy densities of the form\begin{align}
    \label{fwf: levy measure for subordinator}
    \pi^{R}_{r,q,\alpha}(\dif u):=\mathbb{I}_{[0,r]}(u) e^{-qu}u^{-\alpha-1}\dif u + \iota(\dif u),
\end{align}where $r\in(0,\infty],q\geq 0,\alpha\in(0,1)$ and $\iota_r((0,\infty))<\infty$.

Option B (step 1b) offers an alternative: Algorithm A-3
), extending [\cite{Jenkins_2017}, Algorithm 2], to include sWF driven by a rich class of subordinator clocks. 
Algorithm \ref{alg:simulating_sub_dual}
 makes use of a sampling scheme for alternating series of Devroye (\cite{Devroye1986}, Ch. 4).
Denote $b_k^{(t,|\theta|,\psi)}(m):=a_{km}^{|{\theta}|}  e^{-t\psi(\lambda_k)}, k\in\Z_+$, and define the quantity\begin{align}
    \label{eqn Cm}
       {B}_m^{(t,{|{\theta}|},\psi)} &:=\inf\{i\geq 0: {b}_{i+m+1}^{(t,{|{\theta}|,\psi})}(m) < {b}_{i+m}^{(t,{|{\theta}|},\psi)}(m)\}
    \end{align}as well as the following lower and upper sums that will be used to sandwich the sequence $b_k^{(t,|\theta|,\psi)}(m)$:
\begin{align*}
    S_{\boldsymbol{k}}^{-}(M) := \sum_{m=0}^M{\sum_{i=0}^{2k_m+1}{(-1)^i b_{m+i}^{(t,|\theta|,\psi)}(m)}},~~~S_{\boldsymbol{k}}^{+}(M) := \sum_{m=0}^M{\sum_{i=0}^{2k_m}{(-1)^i b_{m+i}^{(t,|\theta|,\psi)}(m)}}.
\end{align*}
The Algorithm is then structurally identical to \cite{Jenkins_2017}'s Algorithm 2.
\begin{algorithm}
\caption{Exact simulation of $\widetilde{q}_k^{\vert{\theta}\vert}(t)$}
\label{alg:simulating_sub_dual}
\begin{algorithmic}[1]
    \State $m \gets 0, \quad k_0 \gets 0, \quad \boldsymbol{k} \gets (k_0)$
    \State Simulate $U \sim \text{Uniform}[0,1]$
    
    \Loop 
        \State Set $k_m \gets \lceil B_m^{(t,{\theta})}/2 \rceil$
        
        \While{$S_{\boldsymbol{k}}^-(m) < U < S_{\boldsymbol{k}}^+(m)$}
            \State $\boldsymbol{k} \gets \boldsymbol{k} + (1, 1, \dots, 1)$ 
        \EndWhile

        \If{$S_{\boldsymbol{k}}^-(m) > U$}
            \State \Return $m$
        \Else 
            \State $\boldsymbol{k} \gets (k_0, k_1, \dots, k_m, 0)$
            \State $m \gets m + 1$
        \EndIf
    \EndLoop
\end{algorithmic}
\end{algorithm}

Algorithm \ref{alg:simulating_sub_dual} works if one can ensure that, for any fixed $m$ and $t$, the sequence of terms in the series  $$\wt q^{\theta,\psi}_m(t):=\sum_{k\geq m} (-1)^{k-m}a_{km}^{|{\theta}|}  e^{-t\psi(\lambda_k)}$$ become, in absolute value, eventually monotone decreasing.
Proposition \ref{prop sampling sWF} identifies conditions on the L\'evy exponent $\psi=\psi(\beta,\pi)$ of the clock for it to happen.

\begin{prop}\label{prop sampling sWF}
    Let $\widetilde{Z}$ be the dual process of a sWF process $\widetilde{{X}}={X}\circ S$, for a subordinator $S$ with Laplace exponent $\psi=\psi(\beta,\pi).$ Assume either $\beta> 0$ or 
 \begin{align}
    \beta=0\ \ \text{and}\ \  \pi(\dif u)\sim u^{-1-\alpha}\dif u,\ \ u\downarrow 0, \ \ \  \alpha\in(1/2,1) .
        \label{eq levycondition}
    \end{align} 
   Let $ {b}_{k}^{(t,{|{\theta}|,\psi})}(m)=a_{km}^{|{\theta}|}  e^{-t\psi(\lambda_k)}.$ Define
    \begin{align*}
        {D}_0^{(t,{|{\theta}|},\psi)}&:=\inf\{k\geq \Big(\frac{1}{t\beta}-\frac{{|{\theta}|}+1}{2}\Big)\vee 0:({|{\theta}|}+2k+1)e^{- t(\psi(\lambda_{k+1})-\psi(\lambda_k))}<1\}.
    \end{align*}
    Then
\begin{enumerate}
    \item ${B}_m^{(t,{|{\theta}|,\psi})}<\infty ~\forall ~m$.
    \item ${b}_k^{(t,{|{\theta}|,\psi})}(m) \downarrow 0$ as $k\rightarrow \infty~\forall~ k\geq m+ {B}_m^{(t,{|{\theta}|})}$.
    \item ${B}_m^{(t,{|{\theta},\psi|})}=0~\forall~m> {D}_0^{(t,{|{\theta}|},\psi)},\beta\neq 0$.
\end{enumerate}
\end{prop}
Figure \ref{fig: sWF tf half half} (below) compares $p_t^{|\theta|}(0.5,\dif u)$ (top left) with ${\wt p}_t^{|\theta|,C}(0.5,\dif u)$ for three choices of subordinator clock. The shape of the transition functions are qualitatively similar. The most noticeable difference is the speed of convergence to the stationary distribution, which in this case would be a Uniform.
Additional figures of the sWF transition function for varying $\theta$ values can be found in the Appendix.

\begin{figure}[H]
    \centering
    \includegraphics[width=0.8\linewidth]
    {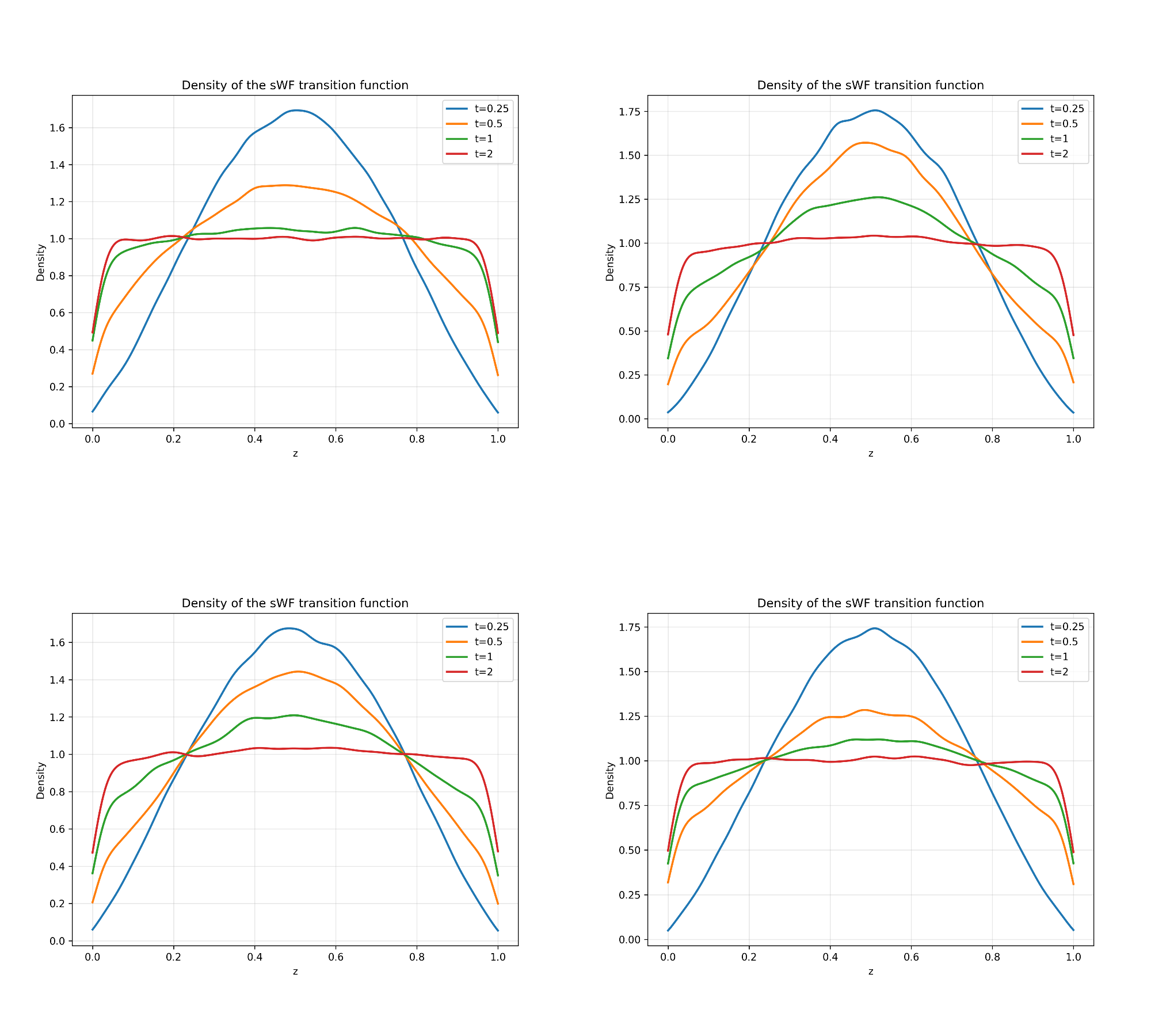}
    \vspace{-0.8cm}
    \caption{$100,000$ samples from the sWF transition function $\widetilde{p}_t^{({\theta},C)}(0.5,\dif z)$ with ${\theta}=(0.5,0.5),~t\in\{0.25,0.5,1,2\}$ and four subordinator clocks $C=S$, (clockwise from topleft): (a) the identity $S(t)\equiv t$, (b) $\alpha$-Stable subordinator with $\alpha = 0.5$; (c) Inverse Gaussian with $\delta=\gamma=1$; (d) Gamma process with $a=b=1$}
    \label{fig: sWF tf half half}
\end{figure}

\newpage

\begin{figure}
\centering
\includegraphics[width=0.9\linewidth]{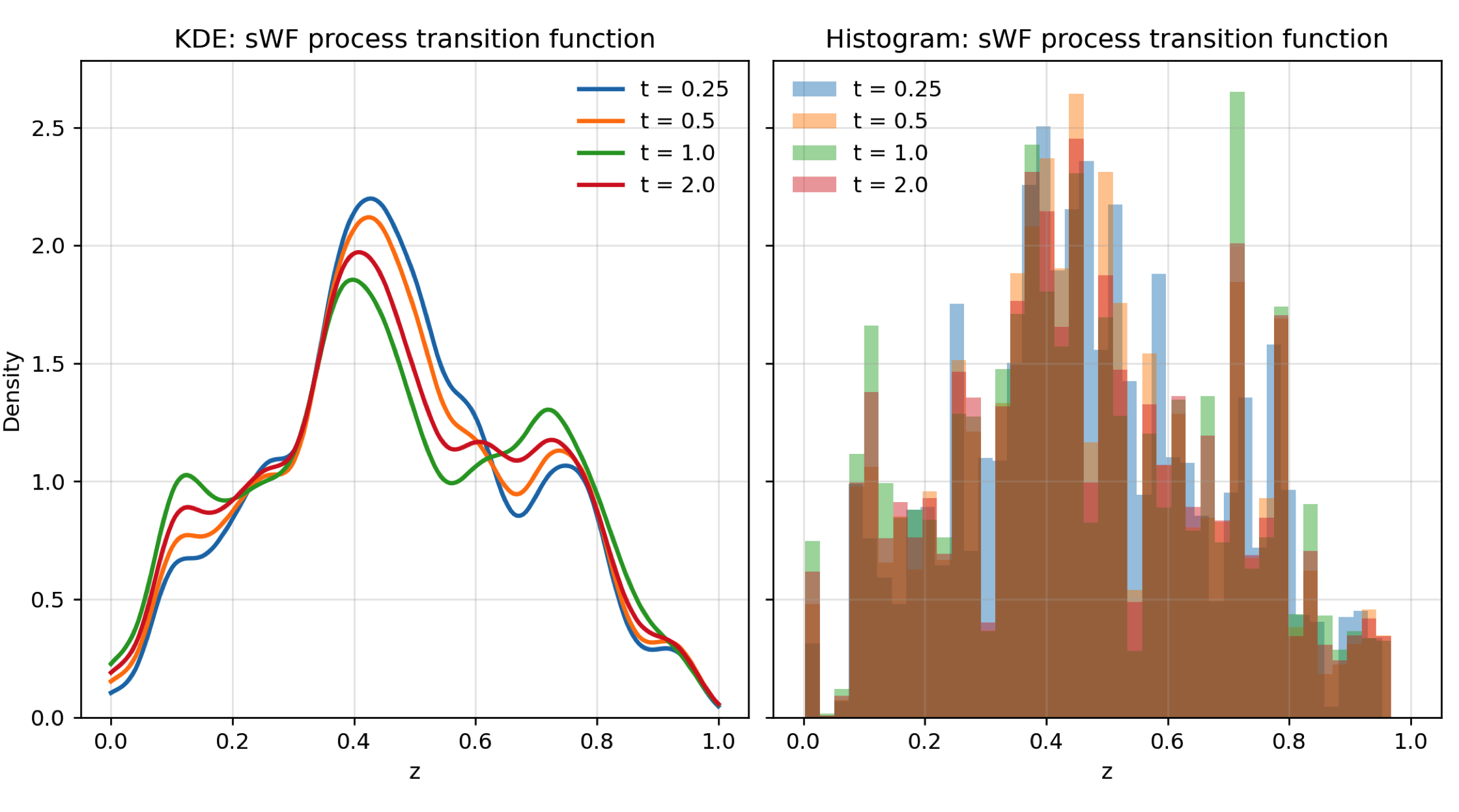}
  \caption{Smoothed density estimate [using Kernel Density Estimation] (left) and Histogram (right) of $10,000$ samples from $\widetilde{p}_t^{([1,1],R)}(0.5,\dif z)$ for time-inputs $t\in\{0.25,0.5,1,2\}$ and inverse-subordinator clock $C=R$ with density $\pi_{1,1,0.5}^{R}(\dif u)$.
}
  \label{fig: fwf tf}
\end{figure}

Figure \ref{fig: fwf tf} confirms what was seen analytically earlier: $\wt p_t^{(|\theta|,R)}(x,\dif u)$ converges more slowly to $\mathcal{D}_{\theta}$ as compared to its diffusive counterpart $p_t^{(|\theta|)}(x,\dif u)$.
\section{Posterior inference and computable filters}\label{sec posterior}
We now turn to the use of sWF processes for inference, leveraging several structural properties in common with WF diffusions, to extend naturally several results on filtering obtained in \cite{Papaspiliopoulos_2014}.
The statistical model of interest is given by the set-up \eqref{The WF Model}, with $ X$ replaced by a sWF process $\widetilde{ X}$: 
 \begin{equation}
         \begin{aligned}
         \label{The Model}
             Y(t_i)\vert \widetilde{ X}(t_i) &\sim f_{\widetilde{ X}(t_i)},\ \ \ i=0,\ldots, N\\
         (\widetilde{ X}(t):t\geq 0)\mid C  &\sim \text{sWF}({\theta},C,\nu^*),~~~C\sim G, 
         \end{aligned}
\end{equation}
where $G$ is the law of the clock $(C(t):t\geq 0)$. For given times $0=t_0<t_1,,\ldots< t_N$, $N\in\mathbb N$, we will denote by $(Y_i,\widetilde{ X}_i)_{i=1,\ldots,N}$ the collection of observation-parameter pairs for each time $t_1,\ldots, t_N,$ respectively. The {\em emission density} $f_{{x}_i}(y)$, i.e. the time-marginal conditional likelihood of $Y_i$ given $\widetilde{{X}_i}={x}_i$, is assumed to be Multinomial$(m_i,  x_i)$ on $\mathbb Z_+^K$, for some integers $m_i: i=0,\ldots, N$. (Analogous results can be obtained trivially under the alternative assumption of $f_{{x}_i}$ being {\em i.i.d.} categorical$({x}_i)$ law on $[K]^{m_i}$). For any vector or sequence $z=(z_0,z_1,\ldots)$, we will use the subsequence notation $z_{l:k}:=(z_l,\ldots,z_k)$.\\
For Markovian sWF models (with subordinator clock), the HMM structure is preserved directly (Section \ref{sec Markov filtering}). For non-Markovian sWF signals (with inverse subordinator clock), we re-formulate the model \eqref{The Model} with $\wt X$ replaced by the augmented Markov signal $V=(X\circ C, C)$, obtaining a new HMM. We will derive in Section \ref{sec frac posterior} the posterior (update and predictive) distributions, given data $Y_{0:i},i=0,...,N$, for 
$V$ as well as for its two marginals $X\circ C$ (our original objective) and for $C$.
 We will show that computability of the filter for 
$X\circ C$ remains ensured by the duality argument, provided that the prior distribution of 
$C$ allows for exact sampling, or that the Laplace transform of the posteriors of $C$ given data are tractable. 
   
\subsection{The first posterior update}
 
 A fundamental observation, holding regardless of the clock choice, is that the first update and prediction steps map finite mixtures of Dirichlet distributions to (finite) mixtures of Dirichlet distributions.
More precisely, define the class
\begin{align*}
    \F := \left\{\sum_{\boldsymbol{m}\in \Z_+^K}{w_{\boldsymbol{m}}\D_{{\theta}+\boldsymbol{m}}(\dif x)}: {\theta} \in \R_+^K,{x}\in\Delta_K,\boldsymbol{w}=(w_{\boldsymbol m})_{\boldsymbol{m}\in\Z_+^K}\in \bigcup_{n=1}^{\infty}\Delta_{n} \right\},
\end{align*} be the set of all finite mixtures of Dirichlet distributions with parameters not less than $\theta$.
\begin{prop}
     \label{cor: base filter steps}
   Let $\widetilde{{X}}$ be a sWF$({\theta},C, \nu^*)$ prior. Assume $\mathcal{L}(\widetilde{{X}}_0)=\nu^*\in \F$. Conditionally on the observation $Y_0=y_0,$ then also $\nu_{0\mid 0:0} \in \F$ and $\nu_{1|0:0}\in \F$. If $\nu^*\in{\F}$, then $\nu_0\in{\F}$ and $\nu_{1|0:0}\in \F$ . 
\end{prop}
Notice that, if one were only interested in proving that, whenever $\nu^*\in\F$, $\nu_{1\vert 0:0}$ consists of a finite mixture of Dirichlet measures (regardless of the sign of its coefficients), one could simply use the expansion \eqref{eqn sub tf polydual} for $\widetilde P_{t}$, in terms of eigenvalues $\Phi_t(\lambda_n)$, multivariate orthogonal polynomial eigenfunctions $Q_{n}({x},{z})$ and stationary distribution $\mathcal{D}_{{\theta}}({x})$: the functions 
    $g_{{\theta}}( x,\boldsymbol{m}+{y_0})$ are polynomials in $x$ of total degree $|\boldsymbol{m}+{y_0}|$.
This means that
$$m(\boldsymbol{m}+{y_0};n, x):=\int_{\Delta_K} g_{{\theta}}( z,\boldsymbol{m}+{y_0})Q_n( z, x){\cal D}(\dif{z}),$$
    where $Q_n$ are as in \eqref{eqn sub tf polydual}, is a polynomial in $ x$ of total degree $|\boldsymbol{m}+{y_0}|$ if $|\boldsymbol{m}+{y_0}|\geq n$, or equals zero if $|\boldsymbol{m}+{y_0}|<n.$ 
    So, using \eqref{fracsg} 
    $$\wt T_{t_1}g_\theta( \cdot,\boldsymbol{m}+{y_0})(x)=\sum_{n\leq|\boldsymbol{m}+\boldsymbol{y_0}|} \Phi_{t_1}(\lambda_n)m(\boldsymbol{m}+{y_0};n, x),$$
gives
    \begin{align}
        \label{eq poly 2pred}
        \nu_{1\vert 0:0}(\dif{x})=\sum_{\boldsymbol{m}}\sum_{n\leq|\boldsymbol{m}+\boldsymbol{y_0}|}[w^*_{\boldsymbol{m}}(\boldsymbol{y}_0) \Phi_{t_1}(\lambda_n)m(\boldsymbol{m}+{y_0};n, x)]\ {\cal D}(\dif{x}),
    \end{align}
     which is a finite sum if $\nu^*\in\F$.
\\
 \subsection{Filtering Markov subordinated Wright-Fisher priors}
\label{sec Markov filtering}
The following recursion is available for any {\em Markov}\   sWF $\widetilde{ X}$, reflecting the model's HMM structure (see \cite{Capp2005InferenceIH, GENONCATALOT2003}. For $i=0,\ldots,N,$ denote the $(i+1)$-th posterior predictive and the $i$-th posterior update distributions, respectively, by $\nu_{i+1\vert 0:i}(\dif {{x}}):=\mathbb{P}[\widetilde{X}_{i+1}\in \dif x\mid Y_{0:i}=y_{0:i}]$ and $\nu_{i\mid 0:i}(\dif {x}):=\mathbb{P}[\widetilde{X}_{i}\in \dif  x\mid Y_{0:i}=y_{0:i}]$. Define $\nu_{0|0:-1}\equiv \nu^*:={\cal L}(\widetilde{{X}}_0),$ the initial distribution of the signal. Then the filter solves the recursion
\begin{align}
    \label{Predictive distributions}
    \nu_{i+1\vert 0:i}(\dif {{x}}) &:=\, \int_{\Delta_K}{\nu_{i\mid 0:i}(\dif  u)\widetilde P_{t_{i+1}-t_i}( u,\dif {{x}})},\ \ \ \tag{\rm P}\\
    \label{Filtering distributions}
    \nu_{i\mid 0:i}(\dif {{x}}) &\propto\, f_{{{x}}}(y_i)\nu_{i\vert 0:i-1}(\dif {{x}}).\ \ \ \ \ \tag{\rm U}
\end{align}
where P stands for \lq\lq prediction step\rq\rq\ and U for \lq\lq update step.\rq\rq \ 
Proposition \ref{cor: base filter steps} shows that such distributions are effectively {\em probabilistic} mixtures in $\F$ and paves the way for the existence of computable filters for Markov sWFs, if it is possible to simulate exactly from the dual law $\wt{q}^{|\bs{\theta}|}_k(t)$. In this case, one only needs to iterate recursively, through \eqref{Filtering distributions}-\eqref{Predictive distributions}, the same steps of the proof of Proposition \ref{cor: base filter steps} to prove the following: 
\begin{cor}\label{cor: markov comput}
    Consider a HMM of the form \eqref{The Model}, where $\wt X\sim sWF({\theta},S,\nu^*)$ for a subordinator clock $S$. If $\nu^*\in\F$, then $\nu_i\in\F$ and $\nu_{i|0:i-1}\in\F$ for all $i\in\Z_+$. 
\end{cor}
\subsubsection{Algorithms}\label{sec alg}
The prediction-filtering and smoothing algorithms for the model \eqref{The Model} are presented in the Appendix Section A-2
as Algorithms 2 and 3, respectively. Note that no sampling from the clock is required as long as the Laplace transforms $\Phi_t(\cdot)$ of $S(t)$ is tractable. Instead, the Markov sWF's dual process is used, which is analytically available in closed form from Proposition \ref{prop t-dual} (in particular, \eqref{eq wtq}-\eqref{eq:wtlambda}).

\subsection{Posterior calculus for non-Markovian sWF priors} \label{sec frac posterior}
For sWF$({\theta},R,\nu^*)$ priors with inverse subordinator clock $C=R,$ the HMM structure of the model no longer holds, so the recursion \eqref{Predictive distributions}-\eqref{Filtering distributions} cannot be used directly.
However it is possible to augment the state space of the signal, to $\Delta_K\times\mathbb R_+$: the augmented signal will be defined as
$$V(t)=( X(R(t)),R(t))=(\wt X(t), R(t)), ~t\geq 0.$$
 The original generalised fractional signal $\wt{ X}$ can be recovered as the marginal distribution of the first coordinate of $V$. In this augmented model, the emission density (conditional likelihood) of the data does not depend on the clock except through $X(R(t))$ i.e. $f_{V(t)}(y)=f_{X(R(t))}(y).$
 While $\wt{ X}$ alone is non-Markovian, the joint process $V$ is a Markov process. 
 \begin{prop}\label{prp vsemigroup}
     The signal $V$ is a Markov process with 
 time-inhomogeneous conditional expectation operator (semigroup) ${\cal T}^V_{s,t}$ given by, for every $s<t$,
 \begin{align}
     {\cal T}^V_{s,t}f( x, r)&=\int_0^\infty \left({\cal T}^{WF}_{u}f(\cdot, r+u)( x)\right) \wt G_{s,t}(r, du),\label{eq: S-sg}
 \end{align}\\
where $\wt G_{s,t}(r,A):=\mathbb{P}\left[R(t)-r\in A\mid R(s)=r\right]$ for all $s,r,y\in\mathbb{R}_+, ~A\in {\cal B}(\mathbb{R}_+)$ and $({\cal T}^{WF}_t)$ is the semigroup of the Wright-Fisher diffusion $ X$.
\end{prop}
The generator ${\cal L}^S$ of $S$ is given, for smooth test functions of the form $h( x, r)=g( x)\psi(r),$ as 
$ 
 {\cal L}^S_t h( x, r)={\cal G}_t\left(\psi(\cdot)\ {\cal T}^{WF}_tg(\cdot)\right)( x, r),
$ 
 where ${\cal G}_t$ is the (time-inhomogeneous) generator of the clock $R$.\\
Let $\bar\nu^*:=\nu^*\times\delta_{0}$ denote the initial distribution of $V$. The model, with the augmented signal $V$, is now again a HMM and the filtering distributions will solve a time-inhomogeneous version of the recursion \eqref{Predictive distributions}-\eqref{Filtering distributions}. Since we now have to distinguish between physical time $t$ and operational time $R(t)$, it will be convenient to incorporate explicitly the time values in the notation for the predictive and update step of the posterior distributions as follows: 
$$\nu^V_{j|0:i}(\cdot \mid t_{0:i})=\mathbb{P}\left[V(t_j)\in \cdot \mid \{Y(t_n)=y_n: n=0,\ldots, i\}\ \right ], \ \ \ j\in\{i,i+1\}$$
(corresponding to \eqref{Filtering distributions} for $j=i$ and to \eqref{Predictive distributions} for $j=i+1$) and similarly
\begin{align*}
&\nu^{WF}_{j|0:i}(\cdot \mid r_{0:j})=\mathbb{P}\left[{ X}(R(t_j))\in \cdot\mid \bigcap_{n=1}^i
\{Y(t_n)=y_n\}\ \cap \bigcap_{n=1}^j\{\ R(t_n)=r_n\}: \right ],\end{align*}
for $j\in\{i,i+1\}.$
Conditionally on $\{R(t_n)=r_n:n=0,\ldots, i+1\},$ $\nu^{WF}_{i|0:i}$ and $\nu^{WF}_{i+1|0:i}$ coincide, respectively, with the $i$-th update and the $(i+1)$-th predictive posteriors, conditional on data $y_{0:i}$ \lq\lq collected at operational times\rq\rq\ $r_1,\ldots,r_i$.\\
The marginal likelihoods of the \lq\lq next sample\rq\rq\ given the previous data, when the signal is given by $V$ or by the Wright-Fisher diffusion $ X$, respectively, will be denoted by
$$m^{V}_{i+1\mid 0:i}(y\mid   t_{1:i+1})=\int_{\Delta_K}f_{ x}(y)\ \nu^{V}_{i+1\mid 0:i}(\dif x),$$
\begin{align*}
    m^{WF}_{i+1\mid 0:i}(y\mid  r_{1:i+1})= \int_{\Delta_K}f_{ x}(y)\ \nu^{WF}_{i+1\mid 0:i}(\dif x\mid r_{1:i+1}),
\end{align*}where it should be recalled that $f_{\boldsymbol{x}}(y)$ does not depend on the clock except through $X(R(t))=x$.
We will denote with $G_{t_1,\ldots,t_i}(\cdot)$ the finite-dimensional marginal distribution of $(R(t_1),\ldots,R(t_i))$ and define the posterior finite-dimensional distributions:
$$G^*_{t_1,\ldots,t_i}(\dif r_{1:i}\mid y_{0:i})=\mathbb{P}\left[R(t_1)\in\dif r_1,\ldots,R(t_i)\in\dif r_{i}\mid Y_{0:i}=y_{0:i}\right]$$
Finally, $m^V(y_{0:i}\mid t_{0:i})$ and $m^{WF}(y_{0:i}\mid u_{0:i})$ will denote the respective joint marginal distributions of the whole data vector $Y_{0:i}$ in either model, arising from repeated composition of $m^{WF}_{i+1\mid 0:i}$ or $m^{WF}_{i+1\mid 0:i},$ respectively.
\begin{prop}\label{prp: i-th v-posterior}
    Let $V=(X\circ R, R)$ for $X\sim WF({\theta},\nu^*)$ and $R$ be an inverse subordinator, independent of $X$. Consider a HMM $(Y_{0:N}, V)$ where the likelihood for each $Y_i=Y(t_i)$, conditionally on $V(t_i)=(x, r)$ is $f_{ x}(y_i)=\mathcal{M}(y_i;  x,{|y_i|)}$. Then, for every $i=1,\ldots$,
    \begin{itemize}
    \item[(i)] the $i$-th predictive distribution is given by
          \begin{align}\label{eq:i-th v-pred}
             \ &\nu^{V}_{i\mid 0:i-1}(\dif ( x, r))
             =\int_{\mathbb{R}_+^{i-1}}\nu^{WF}_{i\mid 0:i-1}(\dif  x\mid r_{{1:i}})\G^*_{t_{1:i-1}}(\dif  r_{1:i-1}\mid y_{0:i-1})\ G_{t_{i-1},t_i}(\dif r_i\mid r_{i-1});
          \end{align}
        \item [(ii)]
 the $(i)$-th update posterior distribution take the form:
    \begin{align}
    \label{eq:ith v-update}
    \nu^V_{i\mid 0:i}(d(x, r))&= \int_{{\mathbb R}^{i-1}_+} G^*_{t_1,\ldots,t_i}(\dif  r_{1:i-1},\dif r\mid y_{0:i})\ \nu^{WF}_{i\mid 0:i}(\dif  x\mid  r_{1:i}).
    \end{align}
          \item[(iii)] The posterior finite-dimensional distributions of the clock $R$ are
    \begin{align}
     G^*_{t_1,\ldots,t_i}(\dif r_{1:i}\mid y_{0:i})&= 
     G^*_{t_1,\ldots,t_{i-1}}(\dif r_{1:i-1}\mid y_{0:i-1}) G_{t_{i-1},t_i}(r_{i-1},\dif r_i)\  \notag\\
     &\times \frac{m^{WF}_{i\mid 0:i-1}(y_i\mid y_{0:i-1}; r_{1:i})}{m^{V}_{i\mid 0:i-1}(y_i\mid y_{0:i-1}; t_{1:i})} \label{eq:i-th frac marg}
    \end{align}
    \end{itemize}
    \end{prop}

   From Proposition \ref{prp: i-th v-posterior}, it becomes clear that the posterior (update and predictive) distributions for $\wt{ X}(t_i)$, arise themselves as mixtures of the corresponding WF diffusion update and filtering distributions, calculated at operational times drawn from the {\em posterior} finite-dimensional distribution $G^*_{\bullet}$ of the clock, given the data. In particular, if the initial marginal distribution of $\wt{X}(0)$ is a Dirichlet (i.e. $\bar\nu^*={\cal D}_{\theta}\times\delta_{\{0\}}$), then both $\nu^V_{i\mid 0:i}(\dif x,\R_+)$ and $\nu^V_{i+1\mid 0:i}(\dif x,\R_+)$ become mixtures of computable posteriors (i.e. mixture of finite mixtures of posterior Dirichlet). Indeed, let us recall that
 for $j\in\{0,1\},$ and for fixed times $r_{1:i+j}$,
    \begin{align*}
\nu^{WF}_{i+j|0:i}(\dif x\mid{r_{1:j}})=
\sum_{\bs m\in C(y_{0:i})}w_{\bs m}(r_{i+j}\mid y_{0:i})\ {\cal D}_{{\theta}+\bs{m}}(\dif x)
\end{align*}
where $C(y_{0:i})\subset\Z_+^K:|C(y_{0:i})|<\infty,$ as seen in \cite{Papaspiliopoulos_2014} and here in Section \ref{sec Markov filtering}. Then, 
by independence of $R$ and $ X$, a simple application of Fubini's Theorem implies the following.
\begin{cor} Assume the same assumptions as Proposition \ref{prp: i-th v-posterior}. Let the initial distribution of $\wt{ X}$ be $\nu^*=\eta_{\theta}\times\delta_{\{0\}}, $ for $\eta\in\F$. Then, for every $i=0,\ldots,N$, and $j\in\{0,1\},$
 \begin{align}
    \nu^V_{i+j\mid 0:i}(\dif x,\R_+)=\sum_{\bs m}w^*_{\bs m}(t_{i+j}\mid y_{0:i}; t_{1:i+j-1})\ {\cal D}_{{\theta}+\bs{m}}(\dif x)
 \end{align}
 where $w^*_{\bs m}(t_{i+j}\mid y_{0:i}; t_{1:i+j-1})=\E_{t_1,\ldots,t_{i+j}}^*\left[w_{\bs m}(R(t_{i+j})\mid y_{0:i})\right]$ with $\E_{t_1,\ldots,t_{i+j}}^*$ being the expectation operator under the distribution $G^*_{t_1,\ldots,t_{i+j}}.$
 \end{cor}
 The tractability of the mixture weights depends on the tractability of $G^*_{\bullet}$ thus, ultimately, on the tractability of the {\em prior} finite-dimensional distributions $G_{t_1,\ldots,t_{i+j}}$ of the clock $R$. Indeed, we remark that both the numerator and the denominator of the fractions appearing in \eqref{eq:i-th frac marg} are finite mixtures of Dirichlet-Multinomial distributions. Thus, regarding \eqref{eq:i-th frac marg} as
 $$
 G^*_{t_1,\ldots,t_i}(\dif r_{1:i}\mid y_{0:i})= G_{t_1,\ldots,t_{i}}(\dif r_{1:i}) \ \frac{m^{WF}_{0:i} (y_{0:i}\mid r_{1:i})}{m^{V}_{0:i}(y_{0:i}\mid  t_{1:i})} \ 
 $$
suggests that, exact algorithms can be obtained for sampling from the posterior $G^*_{t_1,\ldots,t_{i}}$ whenever they exist for $G_{t_1,\ldots,t_i}$. \\
While writing this paper, we have become aware of a new preprint \cite{toaldo25}, providing exact sampling algorithms for $G_{t_1,\ldots,t_i}$ when $R$ is the inverse of a strictly increasing subordinator, encompassing several notable instances (e.g. inverse stable subordinators and inverse gamma subordinators). Thus for at least this family, all the posterior distributions of Proposition \ref{prp: i-th v-posterior} can be sampled from exactly, in finite time.
One simple example is the rejection sampler (Algorithm \ref{alg: rejection sampler for G*}) below, which uses $G_{t_1:t_i}$ as proposal and $m^{WF}_{0:i}$ as rejection probability. Several other exact algorithms can be obtained at least as efficient as Algorithm \ref{alg: rejection sampler for G*}.
Faster approximate methods can be also devised, built upon MCMC or other  algorithms for $G$ and the complexity of the algorithms for $G^*_\bullet$ should be comparable to the complexity of those for $G_{\bullet}$. A full derivation and study of such algorithms is beyond the scope of the present paper. \\

\begin{algorithm}[hbt!]
\caption{Rejection sampler for $G^*_{t_{1:n}}(\dif r_{1:i} | y_{0:i})$}
\label{alg: rejection sampler for G*}
\begin{algorithmic}[1]
\Repeat
    \State Sample $R \sim G_{t_{1:i}}$
    \State Sample $U \sim \mathrm{Uniform}(0,1)$
\Until{$U \le m^{WF}(y_{0:i}|R)$}
\State \Return $R$
\end{algorithmic}
\end{algorithm}

 \bibliographystyle{plain}
 \bibliography{refs}

\newpage

\appendix

\setcounter{equation}{0}
\renewcommand{\theequation}{A-\arabic{equation}}
\renewcommand{\thesection}{A-\arabic{section}}
\renewcommand{\thethm}{\thesection.}
\renewcommand{\theprop}{\thesection}
\renewcommand{\thelem}{\thesection}
\renewcommand{\thealgorithm}{\thesection}
\renewcommand{\thecor}{\thesection}
\setcounter{section}{0}
\setcounter{figure}{0}
\setcounter{algorithm}{0}
\setcounter{lem}{0}
\setcounter{cor}{0}
\setcounter{prop}{0}

\setcounter{page}{1}

\FloatBarrier
    \title{Appendix. Subordinated Wright-Fisher priors}

\section{Additional plots}

\begin{figure}[htp]
\centering
\includegraphics[width=\linewidth, scale=0.3]{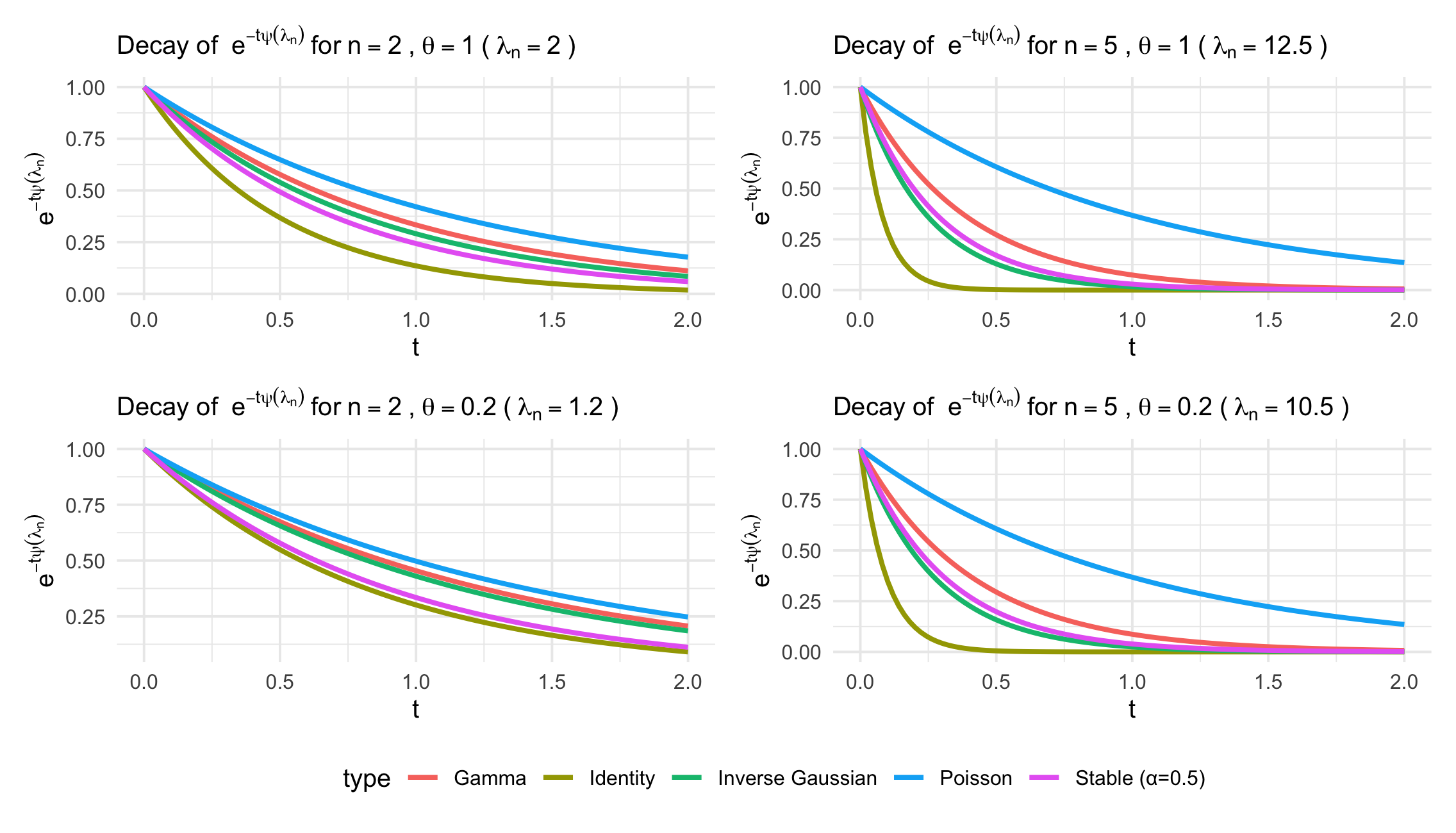}
  \caption{
Laplace transform maps $t\rightarrow \E[e^{-\lambda_n C(t)}]$ across the range $t\in [0,2]$ and for a selection of $(n,\theta)$ and subordinators: (a) $\alpha$-Stable with $\alpha = 0.5$;(b) Poisson $c=1$; (c) Inverse Gaussian, $\delta=\gamma=1$; (d) Gamma $a=b=1$; (e) identity.}
  \label{fig: Laplace Transform Maps}
\end{figure}

\begin{figure}
    \centering
    \includegraphics[width=\linewidth, scale=0.3]{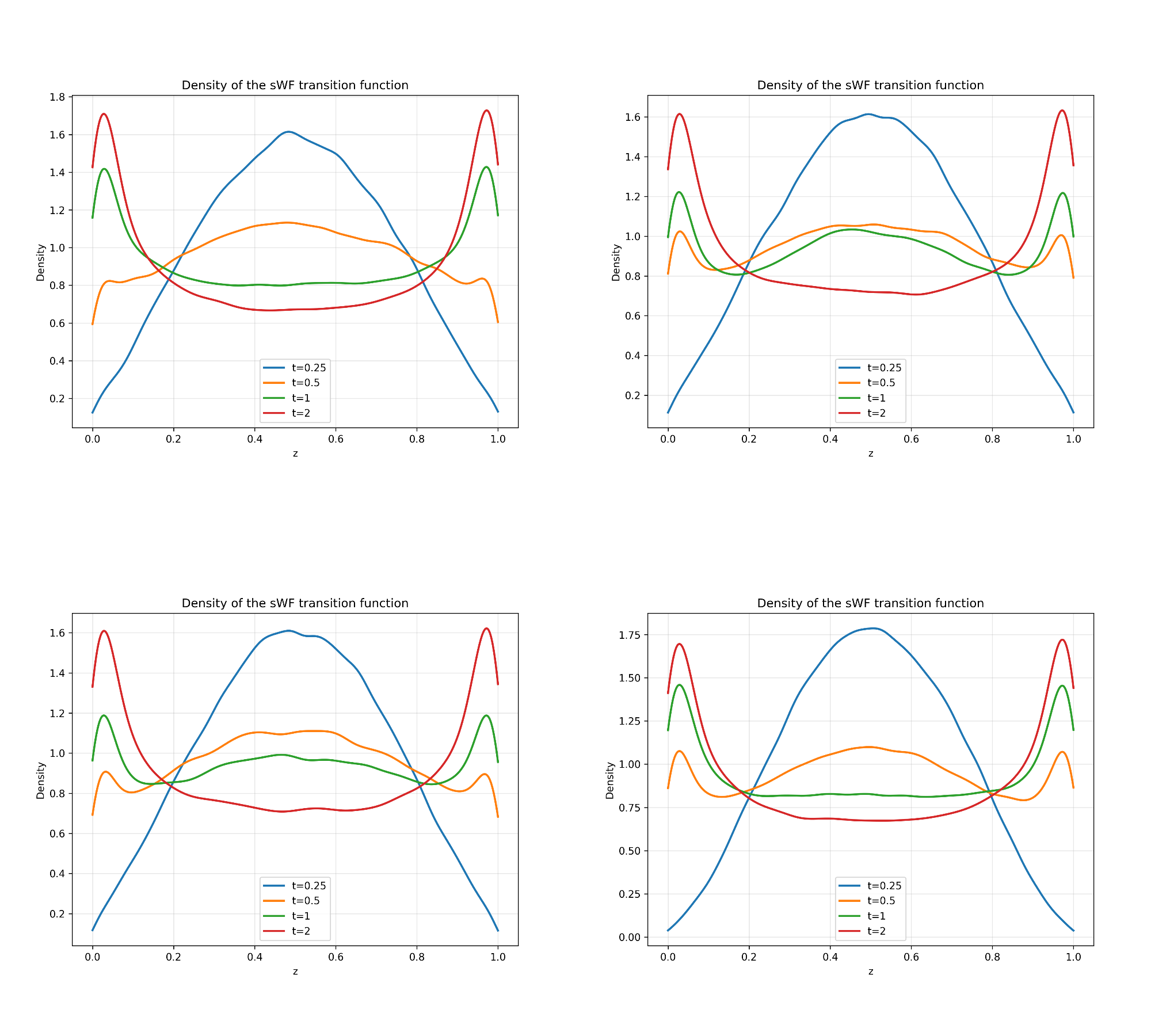}
    \caption{$100,000$ samples from the sWF transition function $\widetilde{p}_t^{({\theta},C)}(x,\dif z)$ with mutation parameters ${\theta}=(1,1)$, initial coordinate $x=1/2$ and time-input $t\in\{0.25,0.5,1,2\}$ where each pane represent a different subordinator clock $C=S$, (clockwise from topleft): (a) the identity $S(t)\equiv t$, (b) $\alpha$-Stable subordinator with $\alpha = 0.5$; (c) Inverse Gaussian with $\delta=\gamma=1$; (d) Gamma process with $a=b=1$.}
    \label{fig: sWF tf 1 1}
\end{figure}

\section{Filtering and smoothing algorithms}

The filtering algorithm for a Markov sWF is presented below in \ref{alg:Kd_filtering}, which makes use of the Bayesian update and prediction steps outlined in Section \ref{sec Markov filtering}.

\begin{algorithm}
\caption{Filtering Algorithm}
\label{alg:Kd_filtering}
\begin{algorithmic}[1]
    \State $\boldsymbol{m}_0 \gets {y}_0$
    \State $W_0 \gets \{1\}$
    \State $W^* \gets \{\widetilde{q}_{\boldsymbol{m}_0 \boldsymbol{k}}(\Bar{t}_0) : \boldsymbol{k} \leq \boldsymbol{m}_0\}$
    
    \For{$i = 1$ \textbf{to} $N$}
        \State \Comment{Update step}
        \State $W_i \gets \left\{ \frac{\boldsymbol{w}_{\boldsymbol{m}}^* \mu_{\boldsymbol{m}}(\boldsymbol{y}_i)}{\sum_{\boldsymbol{k} \leq \boldsymbol{m}_0} w^*_{\boldsymbol{k}} \mu_{\boldsymbol{k}}(\boldsymbol{y}_i)} \;\middle\vert\; \boldsymbol{m} \leq \boldsymbol{m}_{i-1}, \boldsymbol{w}^* \in W^* \right\}$
        \State $\boldsymbol{m}_i \gets \boldsymbol{y}_i + \boldsymbol{m}_{i-1}$
        
        \State \Comment{Prediction step}
        \State $W^* \gets \left\{ \sum_{\boldsymbol{m}_{i-1} \leq \boldsymbol{m} \leq \boldsymbol{m}_i, \boldsymbol{n} \leq \boldsymbol{m}} \boldsymbol{w}_{\boldsymbol{m}}^{(i)} \widetilde{q}_{\boldsymbol{m}\boldsymbol{n}}(\Bar{t}_{i}) : \boldsymbol{w}^{(i)} \in W_i \right\}$
    \EndFor
    
    \State \Return $(\boldsymbol{m}_i)_{0\leq i \leq N}$ and $(W_i)_{0\leq i \leq N}$
\end{algorithmic}
\end{algorithm}
The $i$th smoothing distribution, for $i\in\{0,\ldots,N\}$ is defined as the posterior distribution of $\widetilde{{X}}_i$ conditionally given on data $y_0,\ldots,y_N.$ It can be expressed in terms of the so-called {\em cost-to-go function}, also known as the information filter, and the $i$th filtering distribution $\nu_i$:\begin{align}
    \label{Smoothing distributions}
    \nu_{i\vert 0:n}(\dif {{x}}) &= \frac{p(y_{i+1:n}\vert {{x}}) \nu_{i}(\dif {{x}})}{p(y_{i+1:n}\vert y_{0:i})}.
\end{align}
This is just a version of Bayes theorem, which uses the fact that, when the signal is Markovian, conditionally on $\widetilde{{X}}_i= x$, the distribution of \lq\lq future data, $Y_{i+1:n}$ does not depend on the past observations $y_{1:i}.$ The collection of smoothing distributions is known as the smoothing filter. Computability of the smoothing filter for a Markovian sWF is a direct consequence of Proposition \ref{cor: base filter steps} and Corollary \ref{cor: markov comput}. It extends the known analogous result  on smoothing for WF diffusions, given by Theorem 4 in \cite{kkk}, with the only difference that role of the WF diffusion's dual $Z$ is replaced by the subordinated dual $\widetilde{Z}$
\begin{cor} 
    The model \eqref{The Model} admits a computable smoothing filter.
\end{cor}
We omit the proof and instead refer to the proof of  Theorem 4 in \cite{kkk}, of which it follows exactly the same steps, where the weights are constructed from  $\widetilde{Z}=Z\circ S $ rather than $Z$.
\begin{algorithm}
\caption{Smoothing Algorithm}
\label{alg:smoothing}
\begin{algorithmic}[1]
    \State Compute $(\boldsymbol{m}_i)_{i \geq 1}$ and $(W_i)_{i \geq 1}$ via Algorithm \ref{alg:Kd_filtering}
    \State $\overleftarrow{W}_N \gets \{1\}$
    \State $\overleftarrow{\boldsymbol{m}}_{N|N} \gets {y}_N$
    \State $\overleftarrow{W}^* \gets \{\widetilde{q}_{\overleftarrow{\boldsymbol{m}}_{N|N} \boldsymbol{k}}(\Bar{t}_N) : \boldsymbol{k} \leq \overleftarrow{\boldsymbol{m}}_{N|N}\}$

    \For{$i = N-1$ \textbf{down to} $1$}
        \State \Comment{Backward update step}
        \State $\overleftarrow{W}_i \gets \left\{ \frac{\boldsymbol{w}_{\boldsymbol{m}}^* \mu_{\boldsymbol{m}}({y}_i)}{\sum_{\boldsymbol{k} \leq \boldsymbol{m}_0} w^*_{\boldsymbol{k}} \mu_{\boldsymbol{k}}({y}_i)} \;\middle\vert\; \boldsymbol{m} \leq \overleftarrow{\boldsymbol{m}}_{i+1|i+1:N}, \boldsymbol{w}^* \in \overleftarrow{W}^* \right\}$
        \State $\overleftarrow{\boldsymbol{m}}_{i|i:N} \gets {y}_i + \overleftarrow{\boldsymbol{m}}_{i+1|i+1:N}$

        \State \Comment{Backward prediction step}
        \State $\overleftarrow{W}^* \gets \left\{ \sum_{\overleftarrow{\boldsymbol{m}}_{i-1} \leq \boldsymbol{m} \leq \overleftarrow{\boldsymbol{m}}_i, \boldsymbol{n} \leq \boldsymbol{m}} \boldsymbol{w}_{\boldsymbol{m}}^{(i)} \widetilde{q}_{\boldsymbol{m}\boldsymbol{n}}(\Bar{t}_{i}) : \boldsymbol{w}^{(i)} \in \overleftarrow{W}_i \right\}$
    \EndFor

    \For{$i = 1$ \textbf{to} $N-1$}
        \State $\boldsymbol{M}_{i|0:N} \gets \{(\boldsymbol{m}, \boldsymbol{n}) : 0 \leq \boldsymbol{m} + \boldsymbol{n} \leq \overleftarrow{\boldsymbol{m}}_{i|i+1:N} + \boldsymbol{m}_{i|0:i}\}$
        \State $W_{i|0:N} \gets \left\{ w_{\boldsymbol{1}}^i \propto \sum_{\substack{\boldsymbol{m} \leq \overleftarrow{\boldsymbol{m}}_{i|i+1:N}, \boldsymbol{n} \leq \boldsymbol{m}_{i|0:i} \\ \boldsymbol{m}+\boldsymbol{n} = \boldsymbol{l}}} \overleftarrow{w}_{\boldsymbol{m}}^{(i+1)} w_{\boldsymbol{n}}^{(i)} C_{\boldsymbol{m}, \boldsymbol{n}} : \boldsymbol{1} \in \boldsymbol{M}_{i|0:N} \right\}$
    \EndFor

    \State \Return $(\boldsymbol{M}_{i|0:N})_{1 \leq i < N}$ and $(W_{i|0:N})_{1 \leq i < N}$
\end{algorithmic}
\end{algorithm}

\section{Proofs}
\label{Proof_SM}

\begin{proof}[Proof of Proposition \ref{prop tf sub}]
All the statements of the proposition are well-known to hold for the WF diffusion \cite{griffiths06,Griffiths10}. The proof for general sWF follows simply by replacing 4 (deterministic) with $C(t)$ (random). For \eqref{eqn sub tf polydual} we have
    \begin{align*}
    \widetilde{p}_t({{x}},\dif {u})&=\int_0^\infty\sum_{n=0}^\infty e^{-\lambda_ns} Q_n({x},{u})\mathcal{D}_{{\theta}}(\dif u)\ \zeta^t(\dif s)
    \notag\\
    &=\sum_{n=0}^\infty \left(\int_0^\infty e^{-\lambda_n s} \zeta^t(\dif s)\right)Q_n(x,u)\mathcal{D}_{{\theta}}(\dif u)
     \notag\\
    &=
    \sum_{n=0}^\infty \Phi_t(\lambda_n)Q_n(x,u)\mathcal{D}_{{\theta}}(\dif u),
    \end{align*}
    obtained using again independence between $X$ and $C$ (to be continued), and smoothness of both the Laplace transforms $\Phi_t$ and ${\cal D}_{{\theta}}$. Very much the same approach leads to \eqref{eqn tf subdual}.
\end{proof}

\begin{proof}[Proof of Proposition \ref{prop frac conv eqn}]
      First notice that, from the identity
$$\{R(t)<y\}=\{S(y)>t\}$$ 
one can use integration by parts to derive the well-known formula for the double Laplace transform (see e.g. \cite{Toaldo2015}, Formula (4.8)):
\begin{align}
\Phi(\gamma,\lambda):=\int_0^\infty e^{-\gamma t}\Phi_t(\lambda)\dif t=\frac{\psi(\gamma)}{\gamma(\lambda+\psi(\gamma))}.
    \label{eqn doulelap}
    \end{align}
Using $\int_0^\infty e^{-\gamma t}\kappa(t) dt =1/\psi(\gamma)$ (Formula 2.2,\cite{KSV2012})
and taking the Laplace transform in $t$ of the right hand side of \eqref{eqn gen kin}, one has
    \begin{align*}
        -\lambda\int_0^\infty e^{-\gamma t}\int_0^t\kappa(t-s)\Phi_s(\lambda) \dif s\dif t
        &=-\lambda\int_0^\infty \Phi_s(\lambda)\left[\int_s^\infty e^{-\gamma t} \kappa(t-s)\dif t\right]\dif s\notag\\
        &=-\lambda\frac{\Phi(\gamma,\lambda)}{\psi(\gamma)}
    \end{align*}
    Since $\Phi_0(\lambda)=1$, then taking the time-Laplace transform of the left hand side leads to the identity
    $$\Phi(\gamma,\lambda)=\frac{1}{\gamma}-\lambda\frac{\Phi(\gamma,\lambda)}{\psi(\gamma)}=\frac{\psi(\gamma)}{\gamma(\lambda+\psi(\gamma))}$$
    which coincides with \eqref{eqn doulelap} and the claim follows by uniqueness of the Laplace transform.\\
    Since \begin{align*}
    \frac{1}{\psi(\lambda)}\left[\beta +\int_0^\infty e^{-\lambda y}\bar\pi(y)\dif y \right]=\frac{1}{\lambda},\notag\end{align*} 
    then by Laplace inversion one has 
    $$\beta k(t)+\int_0^t k(t-s) \bar\pi(s) \dif s =1,$$
and \eqref{eqn phi fractional} follows by applying $I^{(\beta,\pi)}_t$ to both sides of \eqref{eqn gen kin}.
    \end{proof}

\begin{proof}[Proof of Proposition \ref{prop sub dual}]
Let ${\cal T}_t$ and ${\cal T}^*_t$ denote the conditional expectation operators associated to the process $X$ (i.e. $T_tf(x)=\mathbb{E}[f(X(t))\mid X(0)=x]$) and to its $h$-dual $Z$, respectively, and let $\zeta^t$ be the law of $C(t)$. Duality means that, for any $(x,y)\in E\times F$ and any $t\geq 0$, 
$${\cal T}_t h(\cdot,y)(x)={\cal T}^*_t h(x,\cdot)(y).$$

Now, if $\widetilde{T}_t$ is the conditional expectation operator associated to the time-changed process $\widetilde X=X\circ C$, we have 
\begin{align}
    \widetilde{{\cal T}}_t h(\cdot,y)(x)&=\int_0^\infty {\cal T}_s h(\cdot,y)(x)\zeta^t(\dif s)\notag\\
    &=\int_0^t {\cal T}^*_s h(x,\cdot)(y)\zeta^t(\dif s)\notag\\
   &=\mathbb{E}\left[h(x, Z(C(t))\mid Z(0)=y\right]=:\widetilde{{\cal T}}_t^* h(x,\cdot)(y).\label{dualderiv}
\end{align}

\end{proof}
\begin{proof}[Proof of Proposition \ref{prop t-dual}]
From Proposition \ref{prop sub dual} (see in particular \eqref{dualderiv}), the dual $Z$, time-changed by $C$ must have a conditional distribution of the form 
$$\mathbb{E}\left[{q}_{|\boldsymbol{m}|, |\boldsymbol{l}|}(C(t))\right]=\left(\sum_{j=|\boldsymbol{l}|}^{|\boldsymbol{m}|}\mathbb{E}\left[e^{-{\lambda}_j C(t)}\right]{(-1)}^{j-|\boldsymbol{l}|}a_{j,|\boldsymbol{l}|}^{(|{\theta}|,|\boldsymbol{m}|)}\right)H(\boldsymbol l,\boldsymbol m)=\widetilde{q}_{|\boldsymbol{m}|, |\boldsymbol{l}|}(t)H(\boldsymbol l,\boldsymbol m),$$
where $(q_{n,k}(t))$ are the transition probabilities \eqref{eqn:dual_from_m} of the dual to the original WF diffusion $X$,
giving \eqref{eq wtq}. For every $\boldsymbol{m}$ and $t$ is well-defined as a probability mass function (as a mixture of probability distributions) with support $\{\boldsymbol {l}:\boldsymbol {l}\leq\boldsymbol{m}\}$ (the hypergeometric distribution $H(\boldsymbol {l},\boldsymbol{m})$ assigns zero mass to all points in $\{\boldsymbol {l}\leq\boldsymbol{m}\}$), hence the dual process is non-increasing almost surely. Formula \eqref{eq wtq surv} follows directly from \eqref{eq wtq}, since both $a_{|\boldsymbol{n}|, |\boldsymbol{n}|}^{(|{\theta}|,|\boldsymbol{n}|)}$ and $H(\boldsymbol l,\boldsymbol m)$ are equal to 1. \\
If $C=S$ (subordinator case), then $-\log\psi_t(\lambda_n)=\widetilde{\lambda}_n$ has the form \eqref{eq:wtlambda} by the L\'evy-Khintchine representation. Since the time-changed dual process is Markovian, its transition rates \eqref{eqn subrates}-\eqref{eqn subrates 1} can be obtained from $$\frac{\dif}{\dif t} \mathbb{P}[\widetilde{Z}(t)=\boldsymbol{l}\mid \widetilde{Z}(0)=\boldsymbol{n}] =\widetilde{\lambda}_{|\boldsymbol{n}|,|\boldsymbol{k}|}H(\boldsymbol{k},\boldsymbol{n}),$$
by exchanging differentiation and summation in $\widetilde{q}_{|\boldsymbol{n}|, |\boldsymbol{k}|}(t)$.\\
If $C=R$ (inverse of $S$), equations \eqref{eq wtq inv} and \eqref{eq wtq surv} follow from the integro-differential equations for $t\mapsto \Phi_t$ stated in Proposition \ref{prop frac conv eqn}. Using equation \eqref{eqn phi fractional} finally implies, after exchanging the order of summation and (generalised) differentiation $I^{(\beta,\pi)}_t$,
 \begin{align*}
   I^{(\beta,\pi)}_t{\widetilde{q}}^{ *}_{\boldsymbol{n}, \boldsymbol{k}}(t)\vert_{t=0}&=\left(\sum_{j=|\boldsymbol{k}|}^{|\boldsymbol{n}|} (-1)^{j-|\boldsymbol{k}|+1}a^{(|{\theta}|,|\boldsymbol{n}|)}_{j,|\boldsymbol{k}|}\ \lambda_j\Phi_t(\lambda_j) \vert_{t=0}\right)H(\boldsymbol k,\boldsymbol n)\\
   &=\left(\sum_{j=|\boldsymbol{k}|}^{|\boldsymbol{n}|} (-1)^{j-|\boldsymbol{k}|+1}a^{(|{\theta}|,|\boldsymbol{n}|)}_{j,|\boldsymbol{k}|}\ \lambda_j\right)H(\boldsymbol k,\boldsymbol n)\notag\\
   &=\left(\frac{\dif}{\dif t} q^{}_{\boldsymbol{n}, \boldsymbol{k}}(t)\vert_{t=0}\right)H(\boldsymbol k,\boldsymbol n),
 \end{align*}
where $q^{}_{\boldsymbol{n}, \boldsymbol{k}}(t)$ are the transition probability of the pure-death block-counting process $|Z(t)|$ dual to the underlying diffusion $X$ and \eqref{bcprates} implies \eqref{fracrates}, completing the proof.
\end{proof}

\begin{proof}[Proof of Proposition \ref{prop sampling sWF}]
Proposition 1 in \cite{Jenkins_2017} covers the case $\Psi(\lambda)=\lambda$, which provides a useful shortcut to the analogous result for general sWF duals.\\
Note that
\begin{align}
    \label{eqn ratio}
    \frac{{b}_{k+1}^{(t,{|{\theta}|,\psi})}(m)}{b_{k}^{(t,{|{\theta}|,\psi})}(m)} &= \frac{{|{\theta}|}+m-k-1}{k-m+1}\cdot \frac{{|{\theta}|}+2k+1}{{|{\theta}|}+2k-1} e^{-t(\psi(\lambda_{k+1})-\psi(\lambda_k))}\notag\\
    &=:f_m^{|{\theta}|}(k)e^{-t(\psi(\lambda_{k+1})-\psi(\lambda_k))}.
\end{align}The exponential term on the right-hand side is decreasing in $k$, since $\psi(\lambda_{k+1})>\psi(\lambda_k)$. The term $f_m^{|{\theta}|}(k)$ does not depend on the choice of subordinator $S$, i.e. on $\psi$. It has been shown in \cite{Jenkins_2017} (Proof of proposition 1) that $(f_m^{|{\theta}|})'(k)<0$ for sufficiently large $k$. Then, $\frac{b_{k+1}^{(t,{|{\theta}|})}(m)}{b_{k}^{(t,{|{\theta}|})}(m)}$ is monotonically decreasing in $k$ to $0$ for large $k$ and, in particular, (see \cite{Jenkins_2017}, formulae (26)-(27)) this is the case for $k>(\sqrt{2(m-1)+{|{\theta}|}}-{|{\theta}|})/2$ and, since
\begin{align*}
    \frac{(\sqrt{2(m-1)+{|{\theta}|}}-{|{\theta}|})}{2}    <m.
\end{align*} then $(f_m^{|{\theta}|})'(k)<0$ for all $k\geq m$. 
 Furthermore \begin{align*}
    f_m^{|{\theta}|}(k) < f_k^{|{\theta}|}(k)=({|{\theta}|}+2k+1).
\end{align*}Therefore, once $\frac{b_{k+1}(k)}{b_k(k)}<1$ it is guaranteed that $\frac{b_{i}(m)}{b_i(m)}<1$ for all $i\geq k+1$ and $j\geq k$. Now, for $B_m^{(t,{|{\theta}|})}=0$ to hold, we need $k\geq m$ sufficiently large such that\begin{align}
   L(k):= ({|{\theta}|}+2k+1)e^{-t(\psi(\lambda_{k+1})-\psi(\lambda_k))} < 1.\label{desired}
\end{align}
For every $\lambda\geq 0$, 
\begin{align*}
    \psi(\lambda)=\beta \lambda+ \int_{0+}^\infty{(1-e^{-u\lambda})\pi(\dif u)}\geq \beta\lambda,
\end{align*}
hence one can establish
\begin{align}
    L(k)&=({|{\theta}|}+2k+1)e^{-t(\psi(\lambda_{k+1})-\psi(\lambda_k))} \notag\\
    &= ({|{\theta}|}+2k+1)e^{-t\beta (\frac{2k+{|{\theta}|}}{2})}e^{-t\int_{0+}^\infty{(e^{-\lambda_k u}-e^{-\lambda_{k+1}u})\pi(\dif u)}} \notag\\
    &=({|{\theta}|}+2k+1)e^{-t\beta (\frac{2k+{|{\theta}|}}{2})}e^{-t\int_{0^+}^\infty e^{-\lambda_{k}u}\left(1-e^{-\frac{2k+|\theta|}{2}}\right)\pi(\dif u)}\label{lnewbound}.
\end{align}
We need to show that $L$ is eventually decreasing for large values of $k$, i.e. the exponential terms eventually decrease faster than the linear increase of $({|{\theta}|}+2k+1)$. Now, if $\beta>0$, the left-hand side of \eqref{lnewbound} is bounded above by
$$({|{\theta}|}+2k+1)e^{-t\beta (\frac{2k+{|{\theta}|}}{2})} $$
and it follows
from Proposition 1 in \cite{Jenkins_2017} that $L$ eventually becomes decreasing for sufficiently large $k$.\\
For the case $\beta=0$, we assume \eqref{eq levycondition}. Consider the function
$$h(k):=\int_{0^+}^\infty e^{-\lambda_{k}u}\left(1-e^{-\frac{2k+\theta}{2}u}\right)\pi(\dif u),$$
and treat $L(k)=({|{\theta}|}+2k+1)e^{-h(k)t}$ as a continuous function in its argument $k$. Taking the first derivative, one has
\begin{align}
   L^\prime(k)&=2e^{-h(k)t} -({|{\theta}|}+2k+1)h^\prime(k)te^{-h(k)t}\notag\\
   &=e^{-h(k)t}\left[2-({|{\theta}|}+2k+1)h^\prime(k)t\right].\label{eq Lprime}
\end{align}
Now, 
\begin{align*}
h^\prime(k)&=\int_0^\infty ue^{-\lambda_k u}\left[-\frac{2k+|\theta|-1}{2}\left(1-e^{-\frac{2k+|\theta|}{2}u}\right)+
e^{-\frac{2k+|\theta|}{2}u}\right]\pi(\dif u)\\
&=:\int_0^\infty ue^{-\lambda_k u}B(k,u)\pi(u)\dif u.
\end{align*}

Define $M_k:=2\log k$ and $I_k:=(0,M_k/\lambda_k))$, $I_k^c:=[M_k/\lambda_k,\infty).$ Then split the integral $h^\prime(k)$ into the sum of the integrals over these two regions.
First, consider the tail region $I_k^c$. Since $u\geq M_k/\lambda_k$, then $e^{-\lambda_ku}\leq e^{-M_k}=k^{-2}.$ Moreover,
$$|B(k,u)|\leq \frac{2k+|\theta|-1}{2}+1\sim \gamma k$$
for some constant $\gamma,$ so
\begin{align}
    \Big\vert\int_{I^c_k} ue^{-\lambda_k u}B(k,u)\pi(u)\dif u \Big\vert\leq \delta k^{-1}\int_0^\infty u\pi(u)\dif u\sim k^{-1}\to 0,\ \ \text{as } k\to\infty
\end{align}
therefore the contribution of this integral is negligible for large $k$.\\
Now on $I_k$, $u\leq M_k/\lambda_k\sim k^{-2}\log k.$ We can Taylor expand the integrand and use
$$1-e^{-\frac{2k+|\theta|}{2}u}=\frac{2k+|\theta|}{2}u+{\cal O}(k^2u^2),$$
to obtain
\begin{align}B(k,u)\approx 1-\left(1+\frac{2k+|\theta|-1}{2}\right)\frac{2k+|\theta|}{2}u+{\cal O}(k^2u^2).\label{smallu}\end{align}
Thus, changing variable $y=\lambda_ku,$ $\dif u=\lambda_k^{-1}\dif y,$ we have
\begin{align*}
 ue^{-\lambda_k u}B(k,u)\pi(u)\dif u
&\approx\frac{y}{\lambda_k^2}e^{-y}\left[1-\left(1+\frac{2k+|\theta|-1}{2}\right)\frac{2k+|\theta|}{2}\frac{y}{\lambda_k}\right]\pi\left(\frac{y}{\lambda_k}\right)\dif y\\
&\sim \frac{y}{\lambda_k^2}e^{-y}[1-y]\pi\left(\frac{y}{\lambda_k}\right)\dif y,\ \ \ \text{for large $k$.}
\end{align*}
Using the assumption \eqref{eq levycondition}, for small $y$,$\pi(y/\lambda_k)\sim c [y/\lambda_k]^{-(1+\alpha)},$ so
we can write
\begin{align*}
    h^\prime(k)&\approx \lambda_k^{\alpha-1}\int_0^\infty 
  y^{-\alpha}e^{-y}[1-y]\dif y\\
  &=\lambda_k^{\alpha-1} (\Gamma(1-\alpha)-\Gamma(2-\alpha))\\
  &=\lambda_k^{\alpha-1}\alpha\Gamma(1-\alpha)\\
  &\sim \alpha\Gamma(1-\alpha)k^{2\alpha-2}>0\ \ \ \text{for large $k$},
\end{align*}
where we have used the known property of the Gamma function: $\Gamma(a+1)=a\Gamma(a),$ and the fact that the integral over $I_k^c$ vanisher faster than the integral over $I_k$
Now, turning to the derivative of $L$, this implies
\begin{align}
L^\prime(k)&=e^{-h(k)u}\left[2-(2k+\theta-1)uh^\prime(k)\right]\notag\\
&\sim e^{-h(k)u}\left[2-2u \alpha\Gamma(1-\alpha)k^{-1+2\alpha}\right]
    \end{align}
Therefore, if $\alpha>1/2,$ $k^{-1+2\alpha}\to\infty$ and $L^\prime (k)<0$ eventually, and this proves the claim.
Notice that if $\alpha\leq 1/2,$ then $L^\prime(k)>0$ eventually, so under the assumption that $\pi(u)\sim  u^{-(1+\alpha)}$ near zero, the condition $\alpha\in(1/2,1)$ is also sufficient.
\end{proof}

\begin{proof}[Proof of Proposition \ref{cor: base filter steps}]
    If $\nu^*=\sum_{\boldsymbol{m}}w_{\boldsymbol{m}}\D_{{\theta}+\boldsymbol{m}}\in{\cal F},$ then $\nu_{0|0:0}$ is a sum of the form
    $$ \nu_{0\mid 0:0}(\dif \boldsymbol x)=\sum_{\boldsymbol m}w^*_{\boldsymbol{m}}({y}_0)g_{{\theta}}(\boldsymbol x,\boldsymbol{m}+{y_0}){\cal D}_{{\theta}}(\dif{x})$$
    for non-negative coefficients 
    \begin{align*}
        w^*_{\boldsymbol{m}}({y}_0):=w_{\boldsymbol{m}}\frac{\prod_{j=1}^K (\theta_j+m_j)_{(y_{0,j})}}{(|{\theta}+\boldsymbol{m}|)_{(|y_0|)}},
    \end{align*}where $y_{0,j}$ is defined as the $j$th coordinate of the vector $y_{0}$.
    
    The function $g_{\boldsymbol{\theta}}$ is precisely the duality function introduced in \eqref{eqn the h function}. So $\nu_0\in\bar\F$. This is the well-known Bayesian conjugacy property of mixtures of Dirichlet distributions likelihoods\cite{Ant74}.
That $\nu_{1\mid0:0}$ is also in $\F$ follows from reversibility, for which
we have:
  \begin{align*}
      \nu_{1\vert 0:0}(\dif{x})&=\,\sum_{\boldsymbol{m}}w^*_{\boldsymbol{m}}({y}_0)\int_{\Delta_K}g_{{\theta}}( z,\boldsymbol{m}+{y_0}){\cal D}_{{\theta}}(\dif{z})\widetilde P_{t_1}({z},\dif{x})\notag\\
      &=\,\sum_{\boldsymbol{m}}w^*_{\boldsymbol{m}}({y}_0)\left(\int_{\Delta_K}g_{{\theta}}( z,\boldsymbol{m}+{y_0})\widetilde P_{t_1}({x},\dif{z})\right){\cal D}_{{\theta}}(\dif{x})
  \end{align*}
  We can now use duality: the term between parenthesis corresponds to $\E_{x}\left[g(\wt X(t_1);\bs m + y_0)\right]$, hence Proposition \ref{prop t-dual} ensures that the dual process, starting from $\boldsymbol{m}+{y_0}$, can only descend, in time $t_1$, through states within the {\em finite} set $\left\{\boldsymbol{l}\in\mathbb Z_+^K:|\boldsymbol{l}|\leq|\boldsymbol{m}+{y_0}|\right\},$ according to the probabilities ${\widetilde{q}}^{ *}_{\boldsymbol{m}, \boldsymbol{l}}(t)$ given in \eqref{eq wtq}. 
Therefore
\begin{align*}\int_{\Delta_K}g_{{\theta}}( z,\boldsymbol{m}+{y_0})\widetilde P_{t_1}({x},\dif{z})=\E_{m + y_0}\left[g(x;\wt Z(t_1))\right]=\sum_{\left\{\boldsymbol{l}\in\mathbb Z_+^K:|\boldsymbol{l}|\leq|\boldsymbol{m}+{y_0}|\right\}} {\widetilde{q}}^{ *}_{\boldsymbol{m}+{y_0}, \boldsymbol{l}}(t_1)g_{{\theta}}( x,\boldsymbol{l})
\end{align*}
and\begin{align*}
\nu_{1\vert 0:0}(\dif{x})=\sum_{\boldsymbol{m}}w^*_{\boldsymbol{m}}({y}_0)\sum_{\left\{\boldsymbol{l}\in\mathbb Z_+^K:|\boldsymbol{l}|\leq|\boldsymbol{m}+{y_0}|\right\}} {\widetilde{q}}^{ *}_{\boldsymbol{m}+{y_0}, \boldsymbol{l}}(t){\cal D}_{{\theta}+\boldsymbol{l}}(\dif{x}).
\end{align*}

  All coefficients in the mixture are positive whenever $\nu^*\in{\F}$, so $\nu_{1\vert 0:0}\in\F.$ Moreover, if $\nu^*\in\bar\F,$ the outer sum in $\boldsymbol{m}$ is over a finite set, and the inner sum is always over a finite set, therefore $\nu_{1\vert 0:0}\in\bar\F.$
\end{proof}

\begin{proof}[Proof of Proposition \ref{prp vsemigroup}]
    We want to prove, for $s\leq u\leq t$ the semigroup property
    $${\cal T}_{s,u} {\cal T}_{u,t}f(x,r)={\cal T}_{s,t}f(x,r),\ \ x,r\in\Delta_K\times[0,\infty).$$
By definition,
\begin{align}
   {\cal T}^V_{s,u} {\cal T}^V_{u,t}f(x,r)&=\int_0^\infty {\cal T}^{WF}_y({\cal T}^V_{u,t}f)(\cdot, r+y)(x)\wt G_{s,u}(r,\dif y)  \notag\\
   &=\int_0^\infty {\cal T}^{WF}_y\left(\int_0^\infty {\cal T}^{WF}_zf(\cdot,r+y+z)(x)\wt G_{u,t}(r+y,\dif z)\right)\wt G_{s,u}(r,\dif y)\notag\\
   &=\int_0^\infty {\cal T}^{WF}_{y+z}f(\cdot,r+(y+z))(x)\int_0^\infty\wt G_{u,t}(r+y,\dif z)\wt G_{s,u}(r,\dif y),\label{eq gsem}
\end{align}
where the last equality follows from Fubini and from the semigroup property of the (time-homogeneous) operator $({\cal T}^{WF}_t)$. But
$$\int_0^\infty\wt G_{u,t}(r+y,\dif z)\wt G_{s,u}(r,\dif y)=\wt G_{s,t}(r,\dif z),$$
Thus the right-hand side of \eqref{eq gsem} is equal precisely to ${\cal T}_{s,t}f(x,r).$
\end{proof}

\begin{proof}[Proof of Proposition \ref{prp: i-th v-posterior}]
We begin with a Lemma proving the proposition for the case $i=1$. The proof for general $i$ will follow by induction.

\begin{lem}
\label{l:1st frac filter}
    For the augmented signal $V$ with law given by \eqref{eq: S-sg} and initial distribution $\bar\nu^*$:
    \begin{itemize}
        \item [(i)]
  the first predictive distribution for $V(t_1)$ given data $Y_0=y_0$ is
    \begin{align}
        \nu^V_{1|0:0}(\dif( x,r)\mid t_1)=G_{0,t_1}(\dif r)\nu^{WF}_{1|0:0}(\dif  x\mid r);
        \label{eq:v-1pred}
    \end{align}
\item[(ii)] the update for $V$ at time $t_1$ given $y_0,y_1$ is
    \begin{align}
        \nu^V_{1|0:1}(\dif( x,r)\mid t_1)
        &=\frac{G_{0,t_1}(\dif r)m_{1|0}^{WF}(y_1\mid y_0; r)}{m^V_{1|0}(y_1\mid y_0; t_1)}\nu^{WF}_{1|0:1}(\dif  x\mid r);\label{eq:1_upd_v}
    \end{align}
    \item[(iii)] The conditional distribution of $R(t_1)$ given $y_{0:1}$ is given by
    \begin{align}
 G^*_{t_1}(\dif r\mid y_{0:1}):=\mathbb{P}[R(t_1)\in\dif r\mid Y_{0:1}=y_{0:1}]=\frac{G_{0,t_1}(\dif r)\ m_{1|0}^{WF}(y_1\mid y_0; r)}{m^V_{1|0}(y_1\mid y_0; t_1)}.\label{eq:1st cond marg}
    \end{align}
      \end{itemize}
\end{lem}

\begin{proof}[Proof of Lemma \ref{l:1st frac filter}]
Equation \eqref{eq:v-1pred} follows simply by independence of $\bs X$ and $R$, and the fact that $R(0)=0$ almost surely. Now, since
    $$ \nu^{WF}_{1|0:1}(\dif x\mid r)=\frac{{\rm Bin}(y_1\mid x;|y_1|)\nu_{1|0:0}^{WF}(\dif x\mid r)}{m^{WF}_{1|0}(y_1\mid y_0; r)}$$
with
    $$m^{WF}_{1|0}(y_1\mid y_0; r)=\int_{\Delta_K} {\rm Bin}(y_1\mid x;|y_1|)\nu_{1|0:0}^{WF}(\dif x\mid r),$$
then from \eqref{eq:v-1pred} follows that
\begin{align}
\nu^V_{1|0:1}(\dif(x,r)\mid t_1)&=\frac{f_{x}(y_1)\ \nu^V_{1|0:0}(d(x, r)\mid t_1)}{\int_{\Delta_K\times\R_+}f_{x}(y_1)\ \nu^V_{1|0:0}(d(x, r)\mid t_1)}
\notag \\
&=\frac{G_{0,t_1}(r)\ m^{WF}_{1\mid 0}(y_1\mid y_0; r)}{m^V_{1|0}(y_1\mid y_0; t_1)}\ \left(\frac{f_{x}(y_1)\nu^{WF}_{1\mid 0:0}(\dif x\mid r)}{m^{WF}_{1\mid 0}(y_1\mid y_0; r)}\right)\notag\\
&=\frac{G_{0,t_1}(r)\ m^{WF}_{1\mid 0}(y_1\mid y_0; r)}{m^V_{1|0}(y_1\mid y_0; t_1)}\ \nu^{WF}_{1|0:1}(\dif x\mid r)
\label{eq:1stpred_v}\end{align}
\end{proof}
 We can now prove the general statement for any $i\geq 1$ by induction. Assume the claims (i)-(iii) hold for $i=k-1$. Since $V$ is a HMM and from \eqref{eq: S-sg},
    \begin{align*}
        \nu^V_{k+1\mid 0:k}(\dif ( x_{k+1}, r_{k+1})\mid t_{1:k+1})&=\int_{\Delta_K\times\mathbb R_+}\nu^V_{k\mid 0:k}(\dif (x_k, r_k)\mid t_{i:k})G_{t_k,t_{k+1}}(r_k, \dif r_{k+1}) P^{WF}_{r_{k+1}-r_{k}}(x_k, \dif x_{k+1}).
    \end{align*}
    From the inductive hypothesis on the update operator, this implies 
    \begin{align*}
        & \nu^V_{k+1\mid 0:k}(\dif (\bs x_{k+1}, r_{k+1})\mid t_{1:k+1})\\
        &=\int_{\Delta_K\times\mathbb R_+}\left(\int_{\mathbb{R}_+^{k-1}} G^*_{t_{1:k}}(\dif r_{1:k-1},\dif r_k\mid y_{0:k})\nu^{WF}_{k\mid0:k}(\dif x_{k}\mid r_{0:k})\right)
        G_{t_k,t_{k+1}}(\dif r_{k+1}\mid r_k ) P^{WF}_{r_{k+1}-r_{K}}(x_k, \dif x_{k+1})\\
        &=\int_{\mathbb{R}_+^{k}}G^*_{t_{1:k}}(\dif r_{1:k}\mid y_{0:k})G_{t_k,t_{k+1}}(\dif r_{k+1}\mid r_k)
        \left(\int_{\Delta_k}\nu^{WF}_{k\mid 0:k}(\dif x_{k}\mid r_{0:k})P^{WF}_{r_{k+1}-r_{K}}(x_k, \dif x_{k+1})\right)\\
        &=\int_{\mathbb{R}_+^{k}}G^*_{t_{1:k}}(\dif r_{1:k}\mid y_{0:k})G _{t_k,t_{k+1}}(\dif r_{k+1}\mid r_k)\ \nu^{WF}_{i+1|0:i}(\dif x_{k+1}\mid r_{1:k+1}),
    \end{align*}
   which proves \eqref{eq:i-th v-pred}. (In the first equality, the quantity between parenthesis is integrated over $r_{1:k-1}$. We have applied Fubini in the second-to-last equality).\\
    Looking now at the update, by Bayes' theorem:
    \begin{align*}
        \nu^{V}_{k+1|0:k+1}(\dif (x_{k+1}, r_{k+1}))&=
        \frac{\nu^V_{k+1\mid:0:k}(\dif (x_{k+1}, r_{k+1})\mid t_{1:k+1})\ f_{x_{k+1}}(y_{k+1})}{m^V_{k+1\mid 0:k}(y_{k+1}\mid y_{0:k})}
    \end{align*}
    Applying \eqref{eq:i-th v-pred},
    \begin{align}
        &\nu^{V}_{k+1|0:k+1}(\dif (x_{k+1}, r_{k+1}))\notag\\
        &=\int_{\mathbb{R}_+^k}\left(\frac{\nu^{WF}_{k+1\mid 0:k}(\dif \bs{x}_{k+1}\mid r_{1:{k+1}})f_{\bs x_{k+1}}(y_{k+1})}{m^V_{k+1\mid 0:k}(y_{k+1}\mid y_{0:k})}\right)H^*_{t_1,\ldots,t_k}(\dif r_{1:k}\mid y_{0:k})H_{t_k,t_{k+1}}(\dif r_{k+1}\mid r_k)\notag\\
        &=\int_{\mathbb{R}_+^k}\nu^{WF}_{k+1|0:k+1}(\dif x_{k+1}\mid r_{1:k+1})\frac{m^{WF}_{k+1\mid 0:k}(y_{k+1}\mid y_{0:k},r_{0:k+1})}{m^V_{k+1\mid 0:k}(y_{k+1}\mid y_{0:k})}\ H^*_{t_1,\ldots,t_k}(\dif r_{1:k}\mid y_{0:k})H_{t_k,t_{k+1}}(\dif r_{k+1}\mid r_k).
        \label{eq:H-version}
        \end{align}
        Since this holds for all $y_{0:k+1}, x_{k+1}, r_{k+1},$ \eqref{eq:H-version} implies that a version of the regular conditional distribution of $(R(t_1),\ldots, R(t_{k+1}))$, given $y_{0:k+1},$ is given precisely by $$H^*_{t_1,\ldots,t_{k+1}} (\dif r_{1:k+1}\mid y_{0:k+1})=\frac{m^{WF}_{k+1\mid 0:k}(y_{k+1}\mid y_{0:k},r_{0:k+1})}{m^V_{k+1\mid 0:k}(y_{k+1}\mid y_{0:k})}H^*_{t_1,\ldots,t_k}(\dif r_{1:k}\mid y_{0:k})H_{t_k,t_{k+1}}(\dif r_{k+1}\mid r_k),$$
        almost surely with respect to the marginal law of $Y_{0:k+1}.$ This proves \eqref{eq:ith v-update} as well as \eqref{eq:i-th frac marg}.\\
        
\end{proof}

\end{document}